\documentclass[12pt]{amsart}
\usepackage{amssymb,amscd,amsfonts,mathabx,mathrsfs}
\usepackage[arrow,matrix,curve]{xy}
\usepackage{fullpage}
\usepackage{graphicx}
\usepackage{slashed}

\newcommand\dirac{\partialslash}

\newcommand\R{\mathbb{R}}

\newcommand\ZA{\mathbb{Z}}
\newcommand\Z{\mathbb{Z}}

\newcommand\C{\mathbb{C}}

\usepackage{color}
\definecolor{darkgreen}{cmyk}{1,0,1,.2}
\definecolor{m}{rgb}{1,0.1,1}
\definecolor{green}{cmyk}{1,0,1,0}

\definecolor{test}{rgb}{1,0,0}
\definecolor{cmyk}{cmyk}{0,1,1,0}

\newcommand\End{\operatorname{End}}
\newcommand\Index{\operatorname{Index}}
\newcommand\Tr{\operatorname{Tr}}

\newcommand\Rank{\operatorname{Rank}}

\newcommand\ind{\operatorname{ind}}

\newcommand\Ind{\operatorname{Ind}}

\newcommand\rank{\operatorname{rank}}

\newcommand\SF{\operatorname{SF}}
\renewcommand\sf{\operatorname{sf}}
\newcommand\ch{\operatorname{ch}}

\newcommand\E{\mathcal E}
\newcommand\maE{\mathcal E}
\newcommand\maL{\mathcal L}

\newcommand\maK{\mathcal K}
\newcommand\maH{\mathcal H}
\newcommand\maQ{\mathcal Q}
\newcommand\maP{\mathcal P}
\newcommand\fH{\mathfrak H}
\newcommand\fR{\mathfrak R}

\newcommand\cP{\mathcal P}

\newcommand\cS{\mathcal S}
\newcommand\maS{\mathcal S}

\theoremstyle{plain}
\newtheorem{theorem}{Theorem}[section]
\newtheorem{lemma}[theorem]{Lemma}
\newtheorem{proposition}[theorem]{Proposition}
\newtheorem{corollary}[theorem]{Corollary}

\theoremstyle{definition}
\newtheorem{definition}[theorem]{Definition}
\newtheorem{definition*}{Definition}

\theoremstyle{remark}
\newtheorem{remark}[theorem]{Remark}

\newtheorem{remarks*}{Remarks}

\begin{document}

\title[Positive scalar curvature and twisted eta invariants ]
{Spectral sections, Twisted rho invariants\\ and positive scalar curvature}
%\footnote{\today}
\author{Moulay Tahar Benameur}
\address{Institut I3M, UMR 5149,
 Universit\'e Montpellier 2 et CNRS,
 B\^at. 9 de Math\'ematiques, Campus du Triolet,
Montpellier, France}
\email{benameur@math.univ-montp2.fr}

\author{Varghese Mathai}
\address{Department of Mathematics, University of Adelaide,
Adelaide 5005, Australia}
\email{mathai.varghese@adelaide.edu.au}

\begin{abstract}
We had previously defined in \cite{BM-JGP},  the rho invariant $\rho_{spin}(Y,\E,H, g)$ 
for the twisted Dirac operator $\dirac^\E_H$
on a closed odd dimensional Riemannian spin manifold $(Y, g)$, acting on sections of a flat hermitian vector bundle $\E$ over $Y$,
where $H = \sum i^{j+1} H_{2j+1} $ is an odd-degree differential form on $Y$ 
and $H_{2j+1}$ is a real-valued differential form of degree ${2j+1}$. Here we show that it is a conformal invariant of the pair $(H, g)$.
In this paper we express the defect integer $\rho_{spin}(Y,\E,H, g) - \rho_{spin}(Y,\E, g)$ in terms of spectral flows and prove that 
 $\rho_{spin}(Y,\E,H, g)\in \mathbb Q$, whenever $g$ is  a 
 Riemannian metric of positive scalar curvature. In addition, if the maximal Baum-Connes conjecture 
 holds for $\pi_1(Y)$ (which is assumed to be torsion-free), then we show that  $\rho_{spin}(Y,\E,H, rg) =0$ for all $r\gg 0$,  
 significantly generalizing results in \cite{BM-JGP}. These results are proved using the Bismut-Weitzenb\"ock formula, a scaling trick, the technique of  noncommutative spectral sections, and the Higson-Roe approach \cite{HigsonRoe}. 
 \end{abstract}



\date{}
\maketitle

\tableofcontents

%%%%%%%%%%%%%
\section*{Introduction}
%%%%%%%%%%%%%

In an earlier paper \cite{BM-JGP}, we extended some of the results of
Atiyah, Patodi and Singer \cite{APS1, APS2, APS3} 
on the Dirac 
operator on a spin, compact manifold with boundary, to the case 
of the Dirac operator twisted by an odd-degree closed differential form.
Atiyah, Patodi and Singer studied the {\em Dirac operator} $\dirac^\E$ on an odd dimensional, closed, spin manifold,
which is self-adjoint and elliptic, having spectrum in the real numbers. For this 
(and other elliptic self-adjoint operators), they introduced the eta invariant
which measures the spectral asymmetry of the operator and is a 
spectral invariant. Coupling with flat bundles, they introduced the closely related rho invariant,
which has the striking property that it is independent of the choice of Riemannian
metric needed in its definition, when it is reduced modulo $\ZA$. In \cite{BM-JGP}, we generalized the construction of 
Atiyah-Patodi-Singer to the twisted Dirac operator  $\dirac^\E_H$ with a closed, odd-degree differential 
form as flux and with coefficients in a flat vector bundle.

Consider a closed, oriented, { {$(2m - 1)$}}-dimensional Riemannian spin manifold $Y$ with an odd degree differential form satisfying the condition
\begin{equation}\label{Reality}
H = \sum i^{j+1} H_{2j+1} \text{ with }H_{2j+1} \text{ a real-valued differential form 
of degree } 2j+1.
\end{equation} 
(see \cite{BM-JGP} as to why this condition is necessary for the twisted Dirac operator as below,
to map sections of positive spinors to sections of negative spinors). 
{\bf However, in this paper, we will {\em not} assume that $H$ is a closed form}.
Denote by $\E$ a hermitian flat vector bundle over $Y$ with the 
canonical flat connection $\nabla^\E$.
Consider the twisted Dirac operator $\dirac^\E_H=c\circ \nabla^{S\otimes\E,  H}= \dirac^\E + c(H)$ (see the next section for explanations).
Then $\dirac^{ {\E}}_H$ is a self-adjoint
elliptic operator acting on sections of the bundle of spinors tensored with flat bundle 
$\E$, that is $S\otimes \E$. Let $\eta(\dirac^\E_H )$ denote its eta invariant. The twisted (Dirac) rho invariant 
$\rho_{spin}(Y, \E, H, g)$ is defined to be the difference 
$$
\rho_{spin}(Y, \E, H, g) = \frac{1}{2} \left(\dim(\ker(\dirac^\E_H))
+\eta(\dirac^\E_H)\right)
 - \Rank (\E) \frac{1}{2} \left(\dim(\ker(\dirac_H))
+\eta(\dirac_H)\right),
$$
where 
$\dirac_H$ is the same twisted Dirac operator corresponding to the trivial line bundle. The rho invariant is sometimes also 
known in the literature as the relative eta invariant.
Although the eta invariant  $\eta(\dirac^\E_H )$  is a priori only a spectral invariant, we 
will show in Theorem \ref{thm:indept} that the  twisted rho invariant, 
$\rho_{spin}(Y,\E,H, g)$ depends only on the conformal class of the pair $[H,g]$.
In  \cite{BM-JGP}, we also computed it  for 3-dimensional spin manifolds with a 
degree three flux form, showing that it was typically non-trivial. We analysed the special case when $H$ is a closed 3-form, using the Bismut-Weitzenb\"ock 
formula for the square of the twisted Dirac operator,  which in this case has 
no first order terms, to show that 
%whenever $X$ is a closed spin manifold of positive scalar curvature, then
 $\rho_{spin}(Y,\E,H, g)=\rho_{spin}(Y,\E, g)$ for all $|H|$ small enough, whenever $g$ is a 
 Riemannian metric of positive scalar curvature. It is primarily this last theorem that we refine and significantly generalize 
 in this paper. 
 More precisely, as a consequence of the $\mathbb R/\mathbb Z$-index theorem,
 %under the , 
we first show in Corollary \ref{main1} that when $H = \sum i^{j+1} H_{2j+1} $ is an 
odd-degree differential form on $Y$ as above, the defect  $\rho_{spin}(Y,\E,H, g)-\rho_{spin}(Y,\E, g)$ is an integer. 
This allows us to immediately deduce from a result  established by Weinberger \cite{Weinberger} that our twisted rho invariant
$\rho_{spin}(Y,\E, H, g)$ is always a rational number whenever $g$ is  a  
Riemannian metric of positive scalar curvature. Under the additional assumption that the maximal Baum-Connes conjecture 
holds for $\Gamma$, then we can deduce from a result of Higson-Roe that $\rho_{spin}(Y,\E, H, g)$ is always an integer.
In this paper, we are able to refine this result as will be explained next.

As one could expect, the defect integer $\rho_{spin}(Y,\E,H, g)-\rho_{spin}(Y,\E, g)$ turns out to be closely related with  the spectral flow, and indeed with a defect spectral flow associated with the flat bundle $\maE$. More precisely, fixing the metric $g$ we consider the paths $(\dirac_{vH}^\maE)_{0\leq v \leq 1}$ and $(\dirac_{vH}^{\Rank\maE})_{0\leq v \leq 1}$ of elliptic operators with fixed principal symbols. The usual APS spectral flows of these paths is only well defined when the endpoints are invertible. In general, in order to {\em{intrinsically}}  define these spectral flows, we assume that  the $K$-theory index class of $\dirac$ vanishes in $K_1(C^*\Gamma)$. This allows to use an invertible perturbation  of the path $(\dirac_{vH})_{0\leq v \leq 1}$ of operators  in the Mishchenko calculus (i.e. a global spectral section) and to insure the independence of this global perturbation. We then prove our main theorem:\\

{\bf{Theorem}} \ref{Main}
{\em {
\begin{enumerate}
\item Assume that $\ind_a(\dirac)=0$ in $K_1(C^*\Gamma)$. Then for any $\maE$ and $H$ as before, 
$$
\rho_{\rm {spin}} (Y, \maE, H) - \rho_{\rm {spin}} (Y, \maE) = \sf ( (\dirac^\maE_{v H})_{0\leq v \leq 1}) - \sf ((\dirac^{\Rank\maE}_{v H})_{0\leq v \leq 1}). 
$$
\item In the case that the metric $g$ on $Y$  has positive scalar curvature, and for any $\maE$ and $H$ as before 
$$
\rho_{spin}(Y,\E,H, r g) = \rho_{\rm {spin}} (Y, \maE, g) \text{  for all } r\gg 0.
$$ 
In particular, there exists a metric $g'$ in the conformal class of $g$  such that  for any $\lambda\geq 1 $, we have
$
\rho_{\rm {spin}} (Y, \maE,  H, \lambda  g') = \rho_{\rm {spin}} (Y, \maE, g).
$
\end{enumerate}
}}

\medskip

We then deduce the following\\

{\bf{Corollary}} \ref{cor-main}
{\em {Assume that the metric $g$ on $Y$ has positive scalar curvature, that the group $\Gamma$ is torsion free and satisfies the maximal Baum-Connes conjecture, then for any $\maE$ and $H$ as before, 
$$
\rho_{spin}(Y,\E,H, r g) = 0 \text{  for all } r\gg 0.
$$ 
In particular, there exists a metric $g'$ in the conformal class of $g$  such that  for any $\lambda\geq 1 $, we have
$
\rho_{\rm {spin}} (Y, \maE,  H, \lambda  g') = 0,
$
and so $\rho_{\rm {spin}} (Y, \maE,  H,  g') = 0$.
}}\\

The case when $H=0$ was well known, see for instance  \cite{Keswani, HigsonRoe}.

Our approach relies on a combination of the results of Higson-Roe, especially the surgery 6-term exact sequence \cite{HigsonRoe}, and the appropriate use of noncommutative spectral sections. In the process, we also use the Bismut-Weitzenb\"ock formula  in Theorem \ref{thm:bismut}
together with a scaling trick in Proposition \ref{bismut-vanishing} to establish invertibility of the twisted Dirac operator
in a scaled Riemannian metric of positive scalar curvature. This is 
 completely different to the method used in \cite{BM-JGP} where
 the spectral flow technique from  \cite{APS3,Getzler}
and the  local index theorem \cite{Bismut} for closed degree 3 differential form twists,  are exploited. Notice that this latter local index theorem is unknown for higher degree odd forms $H$.

We next give a brief outline of the results in section \ref{sect:spectral} including the proofs of our main Theorem \ref{Main} and Corollary \ref{cor-main}.
As mentioned earlier, we shall follow the approach by Higson-Roe  \cite{HigsonRoe}, who
conceptualised the earlier approach in \cite{Keswani}, by using their theory of analytic structure groups $\maS_\bullet (\Gamma)$, which are the analog
of the structure groups in surgery theory. These
 fit into the Higson-Roe six-term exact sequence,
$$
\cdots \rightarrow K_0(B\Gamma) \stackrel{\mu}{\rightarrow} K_0(C^*\Gamma) \rightarrow \maS_1 (\Gamma) \rightarrow K_1 (B\Gamma) \stackrel{\mu}{\rightarrow} K_1(C^*\Gamma) \rightarrow \cdots
$$
Let $\dirac$ and $\dirac_H$ be the untwisted and the $H$-twisted Dirac operators on the odd-dimensional spin manifold $Y$ with coefficients in the Mishchenko Hilbert modules bundle. Assume that the analytic index class of the Dirac operator $\dirac$ vanishes in $K_1(C^*\Gamma)$,  then by \cite{PiazzaSchick}, the operators $\dirac$ and $\dirac_H$ do admit spectral sections $\maQ$ and $\maQ_H$ which belong to the Higson-Roe algebra. Moreover, they define classes in $\maS_1\Gamma$ which are pre images of the Baum-Douglas common $K$-homology class  $f_*[\dirac]=f_*[\dirac_H]$. Following \cite{DaiZhang, Wu}, one can then define 
the higher spectral flow, $\SF ((\dirac_{v H})_{0\leq v \leq 1}; \maQ, \maQ_H) \in K_0(C^*\Gamma)$ associated with the spectral sections $\maQ$ and $\maQ_H$. We prove that this spectral flow in turn is a pre image in $K_0(C^*\Gamma)$ of the difference $[\maQ_H]-[\maQ]\in \maS_1\Gamma$.

Now, using the main construction from \cite{HigsonRoe}, we get relative traces on $K_0(C^*\Gamma)$ and $ \maS_1(\Gamma) $ such that the following diagram commutes,
$$
\begin{CD}
K_0(B\Gamma) @>\mu>> K_0(C^*\Gamma) @>>>   \maS_1(\Gamma) @>>> K_1(B\Gamma) @>\mu>> K_1(C^*\Gamma) \\
        @VVV          @V\Tr_\E VV   @V\Tr_\E VV   @V\Ind_\E VV              @VVV \\
0   @>>> \Z @>>>   \R @>>> \R/\Z@>>> 0
\end{CD}
$$
where $\Ind_\E$ denotes the relative $\R/\Z$-index. Then almost by definition, one has (see notations later)
$$
\Tr_\E(\SF ((\dirac_{v H})_{0\leq v \leq 1}; \maQ, \maQ_H))
= \sf ((\dirac^\maE_{v H})_{0\leq v \leq 1}; \maQ^\maE, \maQ^\maE_H) ) -\sf ((\dirac^{\Rank\maE}_{v H})_{0\leq v \leq 1}; \maQ^{\Rank\maE}, \maQ^{\Rank\maE}_H) )
$$ 
where $\sf$ denotes scalar spectral flow as in Atiyah-Patodi-Singer \cite{APS3}.
A meticulous examination of the relation between these spectral flows and the usual APS spectral flow on the one hand, and between the image of the class $[\maQ_H]-[\maQ]$ under $\Tr_\maE$ and the rho invariants on the other hand, enabled us to prove our main theorem \ref{Main} as well as Corollary \ref{cor-main}.

In the appendix, we give an alternate proof of this vanishing result on a compact Riemannian spin manifold of positive scalar curvature in the special case of closed degree 3 differential form twists $H$, using the representation variety of $\Gamma$,  following \cite{M92}.

Renowned research on metrics of positive scalar curvature can be found for instance in \cite{GrLaw, Ros, SY, Hi74}.
Viewing the eta and rho invariant of the Dirac operator (in the untwisted case) as an obstruction to the existence of 
Riemannian metrics of positive scalar curvature on compact spin manifolds was done in \cite{M92, M92-2, Weinberger}, obstructions arising 
from covering spaces using the von Neumann trace in \cite{M98} are some amongst many 
papers on the subject  such as by Weinberger {\em op. cit.} and Keswani \cite{Keswani} and Higson-Roe {\em op. cit.} and Piazza-Schick \cite{PiazzaSchick07, PiazzaSchick07-2}. Spectral sections were introduced by Melrose-Piazza in \cite{MelrosePiazza}, also studied by Dai-Zhang \cite{DaiZhang}, and introduced in the context of noncommutative geometry by Wu \cite{Wu} and Leichtnam-Piazza 
\cite{LeichtnamPiazza}.

The twisted de Rham complex was introduced for closed $3$-form fluxes by Rohm-Witten 
in the appendix of \cite{RW} and plays an important role in string
theory \cite{BCMMS,MS03,AS}, since the charges of Ramond-Ramond fields in
type II string theories lie in the twisted cohomology of spacetime.
$T$-duality of the type II string theories on compactified spacetime gives
rise to a duality isomorphism of twisted cohomology groups \cite{BEM, BEM-2}.
The twisted de Rham differential also appears in supergravity  \cite{VanN,GHR} and superstring theory \cite{St86}
via Riemannian connections  with totally skew-symmetric, (de Rham) closed torsion tensor $H$ -
such Riemannian connections have the same geodesics as the Levi-Civita connection.
The twisted analogue of analytic torsion was studied in \cite{MW, MW2,MW3} and the twisted analogue of 
the eta and rho invariants for the signature operator was studied in \cite{BM}.
Twisted Dirac operators,
known also as {cubic Dirac operators} have been studied in representation theory
of Lie groups on homogeneous spaces \cite{Kostant, Slebarski} and 
loop groups  \cite{Landweber}.   

{\em {Acknowledgement.}} The first author is indebted to E. Leichtnam, P. Piazza and I. Roy for many helpful discussions.  
In \cite{BM-JGP}[Theorem 2.6], the statement of the conformal invariance is unfortunately incorrect since the scaling of the form $H$ was omitted. 
We thank Thomas Schick for pointing out this error to us and to the referee for asking us to clarify if the assumption that
$H$ is closed is really needed. The correct form of conformal invariance is now contained in Section \ref{sec:conf}.
V.M. acknowledges acknowledges funding by the Australian Research Council, through Discovery Projects
DP110100072 and DP130103924.

%%%%%%%%%%%%%%%%%%%%%
\section{Conformal invariance and the $\R/\Z$-index theorem}\label{sec:conf}
%%%%%%%%%%%%%%%%%%%%%%

Let $(Y,g)$ be a closed, oriented, { {$(2m - 1)$}}-dimensional Riemannian spin manifold and 
$H = \sum i^{j+1} H_{2j+1} $ an 
odd-degree, differential form on $Y$ where $H_{2j+1}$ is a real-valued differential form 
of degree ${2j+1}$ (see \cite{BM-JGP} as to why this condition is necessary for the twisted Dirac operator as defined below,
to map sections of positive spinors to sections of negative spinors). Denote by $\E$ a hermitian flat vector bundle over $Y$ with the 
canonical flat connection $\nabla^\E$. Let 
\begin{equation}\label{superconn}
\nabla^{\E,  H} = \nabla^\E + H\wedge
\end{equation}
be the superconnection (in the sense of Quillen \cite{Q, MQ}) on the trivially graded bundle $\E$.
Consider the twisted Dirac operator $\dirac^\E_H=c\circ \nabla^{S\otimes\E,  H}= \dirac^\E + c(H)$,
where the superconnection $\nabla^{S\otimes\E, H} = \nabla^S \otimes 1 + 1\otimes \nabla^{\E, H}$.
Then $\dirac^{ {\E}}_H$ is a self-adjoint
elliptic operator. For basic properties of generalized Dirac operators, see \cite{BGV}.

\subsection{Conformal invariance}

 Let $u:Y\to \R$ be a real-valued smooth function on $Y$ and denote by $\hat g=e^{-2u} g$ the conformally equivalent metric on $Y$. 
For any vector field $V$, consider the conformally equivalent vector field $\hat V=e^u V$. Then
$$
\hat g (\hat V, \hat V) = g(V,V).
$$

A similar relation holds for any $k$-form $\omega$, namely denoting by $\hat \omega := e^{-ku} \omega$, we have
$$
\hat g (\hat \omega, \hat \omega) = g(\omega,\omega),
$$
where now by abuse of notation, $g$ and $\hat g$ denote the Riemannian metrics induced on the exterior powers of the cotangent bundle.
% by the dual Riemannian metrics of $g$ and $\hat g$.

We have the Clifford representations
$$
c: (TM, g) \to \End (S) \text{ and } \hat c : (TM, \hat g)\to \End (\hat S),
$$
and the induced ones on the exterior powers of the cotangent bundle. The bundles $S$ and $\hat S$ are isomorphic and a section $\psi$ of $S$ corresponds to a section $\hat \psi$ of $\hat S$ by composing with this isomorphism.

Hijazi \cite{Hijazi} proves that we have for any vector field/differential form
$$
\hat c (\hat \omega) (\hat \psi) = \widehat{ c(\omega)(\psi)}.
$$
The conformal invariance statement now can be written as:

\begin{proposition}\label{prop:conf}
Let $H= \sum_j H_{2j+1}$ be an odd differential form on $M$ and $\dirac^g_H=\dirac^g+c(H)$ act on the sections of $S$. If we set
$$
H_u := \sum_j e^{-(2j+2)u} H_{2j+1} \quad \text{ and }\quad \dirac^{\hat g}_{H_u} = \dirac^{\hat g} + \hat c (H_u): C^\infty (\hat S) \to C^\infty (\hat S).
$$
Then for any spinor $\psi$, 
$$
\dirac^{\hat g}_{H_u} (e^{-\frac{n-1}{2} u} \,\hat \psi) = e^{-\frac{n+1}{2} u} \,\widehat{\dirac^g_H\psi}.
$$
\end{proposition}

\begin{proof}
We may assume that $H$ is homogeneous of degree $k=2j+1$.
Hence
\begin{eqnarray*}
\dirac^{\hat g}_{H_u} (e^{-\frac{n-1}{2} u} \hat \psi) & = & \dirac^{\hat g} (e^{-\frac{n-1}{2} u} \hat \psi) + \hat c (H_u) (e^{-\frac{n-1}{2} u} \hat \psi)\\
& = & e^{-\frac{n+1}{2} u} \widehat{\dirac^g\psi} + e^{-\frac{n-1}{2} u} e^{ku} \hat c (\hat{H_u}) (\hat \psi)
\end{eqnarray*}
But
$$
e^{-\frac{n-1}{2} u} e^{ku} \hat c (\hat{H_u}) (\hat \psi) = e^{(k-\frac{n-1}{2}) u} \widehat{c (H_u) ( \psi)}  =  e^{(k-\frac{n-1}{2}) u} e^{-(k+1)u} \widehat{c (H) ( \psi)}  = e^{-(n+1)/2} \widehat{c (H) ( \psi)}
$$
This completes the proof.
\end{proof}

We obtain the following immediate corollary.

\begin{corollary}[Conformal invariance of the index and nullspaces]
The following identities hold:
\begin{enumerate}
\item $\dim(\ker(\dirac^{\hat g}_{H_u} )) = \dim(\ker(\dirac^g_H))$;

\item $\Index(\dirac^{\hat g}_{H_u} ) = \Index(\dirac^g_H)$.

\end{enumerate}
That is, both the dimension of the nullspace and the index of the twisted Dirac operator $\dirac^g_H$ are conformal invariants of the pair $(H, g)$.
\end{corollary}

Let $\eta(\dirac^\E_H )$ denote the eta invariant of the self-adjoint operator $\dirac^\E_H $. The twisted (Dirac) rho invariant 
$\rho_{spin}(Y, \E, H, g)$ is defined to be the difference 
$$
\rho_{spin}(Y, \E, H, g) = \xi_{spin}(\dirac^\E_H, g) - \Rank (\E) \xi_{spin}(\dirac_H, g)
$$
where
$$
 \xi_{spin}(\dirac^\E_H, g) =  \frac{1}{2} \left(\dim(\ker(\dirac^\E_H))
+\eta(\dirac^\E_H)\right).
 %- \Rank (\E) \frac{1}{2} \left(\dim(\ker(\dirac_H))
%+\eta(\dirac_H)\right),
$$
Here 
$\dirac_H$ is the same twisted Dirac operator corresponding to the trivial line bundle. It is sometimes 
also known in the literature as the relative eta invariant. The following theorem gives the statement of the conformal invariance with respect to the pair $(H,g)$ and corrects an error in \cite{BM-JGP}.

\begin{theorem}[Conformal invariance of the spin rho invariant]\label{thm:indept}
Let $Y$ be a compact, spin manifold of dimension $2m-1$, $\E$, a flat hermitian vector
bundle over $Y$, and $H = \sum i^{j+1} H_{2j+1} $  an 
odd-degree  differential forms on $Y$ and $H_{2j+1}$ is a real-valued differential form 
homogeneous of degree ${2j+1}$. 
Then the spin rho invariant $\rho_{spin}(Y,\E,H, g)$ of the twisted Dirac operator
depends only on the conformal class of the pair $(H,g)$.
% and on the conformal class
%of the Hermitian metric on $\E$. 
%It is a diffeomorphism invariant of $X, \E$ and $H$.
\end{theorem} 

\begin{proof}
Consider the manifold with boundary $X = Y \times [0, 1]$, where the boundary $\partial X
= Y \times \{0\} - Y \times \{1\}$. 
Choose a smooth function $a(t), \, t\in [0,1]$ such that $a(t)\equiv 0$ near $t=0$ and 
$a(t)\equiv 1$ near $t=1$. Consider the metric $h=e^{2a(t) u}(g + dt^2)$ on $X$, 
which is also of product type near the boundary,
and let $\pi : X \to Y$
denote projection onto the first factor. Let $H_{au}$ be the form on $X$, having the property that it restricts to 
$H$ on the boundary component $Y \times \{0\}$ and it restricts to $H_u$ on the boundary component
$Y \times \{1\}$.

\begin{figure}[h]
\includegraphics[height=2in]{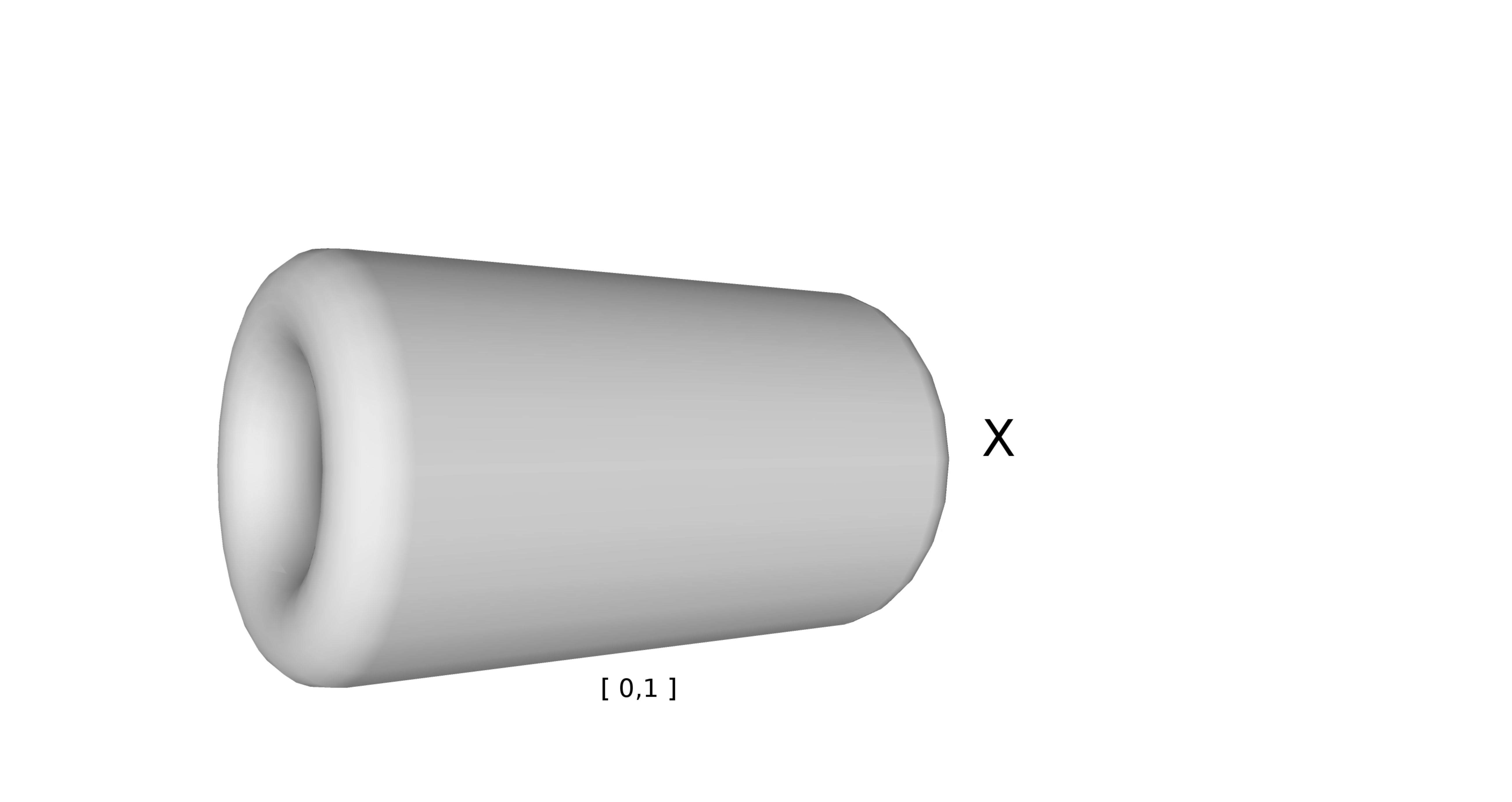}\\
\caption{}
\label{fig:box-metric}
\end{figure}

Applying the index theorem for the twisted Dirac operator, Proposition 2.4 in \cite{BM-JGP}, we get
\begin{equation}\label{eqn:sig1}
{\rm Index}\left(\dirac^{X, \pi^*(\E)}_{H_{au}}, P^+_\E\right) =
\Rank(\E) \int_{X} \alpha_0^H +
\xi_{spin}(\dirac^\E_H, g)
-\xi_{spin}(\dirac^\E_{H_u}, \hat g).
\end{equation}
On the other hand, applying the same theorem to the trivial bundle ${\underline{\Rank (\E)}}$ of rank 
equal to $\Rank (\E)$, we get
\begin{equation}\label{eqn:sig3}
{\rm Index}
\left(\dirac^{X, {\underline{\Rank (\E)}}}_{H_{au}}, P^+\right) 
=  \Rank(\E)\int_{X} \alpha_0^H + 
\Rank (\E) \left[\xi_{spin}(\dirac_H,  g)
-\xi_{spin}  (\dirac_{H_u}, \hat g)\right].
\end{equation}
Subtracting the equalities in \eqref{eqn:sig1} and \eqref{eqn:sig3} above, we get
\begin{equation}\label{eqn:sig5}
\rho_{spin}(Y, \E,H, g) - \rho_{spin}(Y,\E,H_u, \hat g) =
{\rm Index}\left(\dirac^{X, \pi^*(\E)}_{H_{au}}, P^+_\E\right) - {\rm Index}\left(\dirac^{X, {\underline{\Rank (\E)}}}_{H_{au}}, P^+\right).
\end{equation}
Each of the index terms on the right hand side of 
\eqref{eqn:sig5} has been shown in \cite{APS1} to be equal to the $L^2$-index of $\hat X$, which is 
$X$ together with infinitely long metric cylinders glued onto it at $\partial X$, plus a correction term (the dimension of the space of, limiting values of right handed spinors). 
The  $L^2$-index is a conformal invariant of the pair $(H,g)$ by Proposition \ref{prop:conf}, and similarly, the correction term is also a conformal invariant of the pair $(H,g)$. 
It follows by \eqref{eqn:sig5} that 
$$
\rho_{spin}(Y, \E,H, g) = \rho_{spin}(Y,\E,H_u, \hat g),
$$
that is, $\rho_{spin}(Y, \E,H, g)$ is also a conformal invariant of the pair $(H,g)$.
\end{proof}

If we want to use the fact that $\rho_{spin}(Y, \E,H, g)$ is a conformal invariant of the pair $(H,g)$, then we use the notation
$\rho_{spin}(Y, \E,[H, g])$. 

\begin{remark}
Notice that even if $H$ is closed, $H_u$ is not closed unless $u$ is constant. In the sequel, only the case of constant $u$ is thus used when the closeness of $H$ is necessary.
\end{remark}

%%%%%%%%%%%%%%%%%%%
\subsection{Consequences of the $\R/\Z$-index theorem}
%%%%%%%%%%%%%%%%%%%

The following is a variant of Theorem 6.1 in \cite{HigsonRoe}. Notice that we change notation, to be consistent with the rest of the paper.

\begin{theorem}
Let $(Y, \cS, f)$ be a geometric K-cycle for $B\Gamma$ and let $D_H$ be an $H$-twisted Dirac operator for $(Y, \cS, f)$. Then
the reduced $H$-twisted rho invariant 
$$
\rho_{spin}(Y, \E, H) \mod\Z
$$
depends only on the equivalence class  $[(Y, \cS, f)] \in K_1(B\Gamma)$. In particular, we get a morphism,
$$
\ind_\E : K_1(B\Gamma) \longrightarrow \R/\Z
$$
defined by $\ind_\E(Y, \cS, f) = \rho_{spin}(Y, \E, H)  \mod\Z.$
\end{theorem}

\begin{proof}
By  Theorem 6.1 in \cite{HigsonRoe}, there is a well defined map from $K$-homology of  $B\Gamma$ to $\R/\Z$ which is given by the reduced rho. Since the $K$-homology class of $D$ equals that of $D_H$  using Poincar\'e duality on $Y$ and the fact that these operators have the same principal symbol, their reduced rho invariants are the same. Hence the real rho invariants differ by an integer.
\end{proof}

The following is a generalization of Corollary 6.4 in \cite{HigsonRoe}.

\begin{corollary}\label{main1}
The $H$-twisted rho invariant $\rho_{spin}(Y,\E, H,g)\mod\Z$, is independent of the choice of $H$-twisted Dirac operator
$D_H$ associated to the geometric K-cycle  $(Y, \cS, f)$. 

In particular, 
$$
\rho_{spin}(Y, \E, g)=\rho_{spin}(Y, \E, H, g)\mod\Z.
$$
\end{corollary}
One of the goals of this paper is to identify the integer $\rho_{spin}(Y, \E, H, g) - \rho_{spin}(Y, \E, g)$ as the spectral 
flow of explicit paths of twisted Dirac operators. Before that, let us deduce the following corollary of Corollary \ref{main1}
and a result of Weinberger \cite{Weinberger}.

\begin{corollary}
Let $(Y,g)$ be a compact spin manifold with a Riemannian metric
of positive scalar curvature. Then the 
$H$-twisted rho invariant $\rho_{spin}(Y,\E, H,g) \in \mathbb Q$.
\end{corollary}

Let us also deduce the following corollary of  Corollary \ref{main1}
and a result of Higson-Roe \cite{HigsonRoe}.

\begin{corollary}\label{cor:HR}
Let $(Y,g)$ be a compact spin manifold with a Riemannian metric
of positive scalar curvature. Suppose also that  the maximal Baum-Connes conjecture 
 holds for $\pi_1(Y)$. 
Then the 
$H$-twisted  rho invariant $\rho_{spin}(Y, \E, H, g) \in \mathbb Z$.
\end{corollary}

%The rest of the paper is devoted to proving the refinement of Corollary \ref{cor:HR}, that under the same hypotheses
%together with the assumption that $H$ is closed,
%the $H$-twisted rho invariant $\rho_{spin}(Y, \E, H, g)=0$ vanishes. A key step are the results in the following section.

%%%%%%%%%%%%%%%
\section{Scaling trick \& the Bismut-Weitzenb\"ock formulae for odd degree twist}\label{sect:bismut}
%%%%%%%%%%%%%%%%%

Let $Y$ be a closed, { {$(2m - 1)$}}-dimensional Riemannian spin manifold and 
$H = \sum i^{j+1} H_{2j+1} $ an 
odd-degree differential form on $Y$ where $H_{2j+1}$ is a real-valued differential form 
of degree ${2j+1}$. Denote by $\E$ a hermitian flat vector bundle over $Y$ with the 
canonical flat connection $\nabla^\E$.
Let $c(H)$ be the image of $H$ in the sections of the Clifford algebra bundle ${\rm Cliff}(TY)$. 
Then $\nabla^\E + H$ is a superconnection on the trivially graded, flat bundle $\E$ over $Y$. Then 
$ (\nabla^\E + H)^2 =dH $
is the curvature of the superconnection. 

Let $\dirac^\E$ denote the Dirac operator acting on $\E$-valued spinors on $Y$. If $\{e_1, \ldots, e_n\}$ is a local orthonormal 
basis of $TY$, then we have the expression,
$$
\dirac^\E = \sum_{j=1}^n c(e_j) \nabla^{S\otimes\E}_{e_j}, 
$$
where $\nabla^{S\otimes\E} = \nabla^S \otimes 1 + 1\otimes \nabla^{\E}$.
The corresponding expression for the twisted Dirac operator is 
$$\dirac^\E_H =  \sum_{j=1}^n c(e_j) \nabla^{S\otimes\E, H}_{e_j},$$
where the superconnection $\nabla^{S\otimes\E, H} = \nabla^S \otimes 1 + 1\otimes \nabla^{\E, H}$
and the superconnection  $\nabla^{\E, H}$ is defined in equation \eqref{superconn}.
Let $R^g$ denotes the scalar curvature of the Riemannian metric. Then as shown in \cite{Bismut} the twisted 
spinor Laplacian
$$\Delta_H^\E = - \sum_{j=1}^n \left(\nabla^{S\otimes\E}_{e_j} + c(\iota_{e_j} H)\right)^2$$ is a positive operator that 
does not depend on the local orthonormal basis  $\{e_1, \ldots, e_n\}$
of $TY$. Here $\iota_{e_j}$ denotes contraction by the vector $e_j$.

Then the following is a consequence of Theorem 1.1 in \cite{Bismut}.

\begin{theorem}[Bismut-Weitzenb\"ock formulae, Theorem 1.1 in \cite{Bismut}]\label{thm:bismut}
Let $H$ be an odd degree differential form on $Y$ as above.
Then the following identities hold,
$$
\left(\dirac^{g,\E}_{H}\right)^2 = \Delta_H^{g,\E}  + \frac{R^g}{4} + c(dH) 
+ c(H)^2 +  \sum_{j=1}^n c(\iota_{e_j} H)^2 ,
$$
where $R^g$ denotes the scalar curvature of the Riemannian spin manifold $Y$
and where the following holds,
\begin{equation}\label{eqn:weit}
 c(H)^2 +  \sum_{j=1}^n c(\iota_{e_j} H)^2 - c(H^2) = \sum_{j_1<j_2\ldots<j_k, \, k\ge 2}
 (-1)^{\frac{k(k+1)}{2}} (1-k) c((\iota_{e_{j_1}}\iota_{e_{j_2}}\ldots \iota_{e_{j_k}}H)^2).
\end{equation}
where $\{e_j: j=1,2, \ldots, n\}$ is a local orthonormal frame.
\end{theorem}

As a corollary of Theorem \ref{thm:bismut}, one has the following 
%well known result (Theorem 1.3 \cite{Bismut} and Theorem 6.2 of \cite{AF}), one has the following 
special Bismut-Weitzenb\"ock formula

\begin{theorem}[Theorem 1.3 in \cite{Bismut}, see also \cite{AF}] \label{thm:weit}
Let $H$ be a differential 3-form on $Y$. 
Then 
$$
\left(\dirac^{g,\E}_{H}\right)^2 = \Delta_H^{g,\E}  +  \frac{R^g}{4} - 2 \vert H\vert^2_g + c(dH)
$$
where $R^g$ denotes the scalar curvature of the Riemannian spin manifold $Y$ and $|H|$ the length of $H$.
\end{theorem}

We next establish the invertibility of the twisted Dirac operator under the hypothesis of positive scalar curvature. Typically, this is achieved directly from Weitzenb\"ock type formulae. But this is insufficient in our case, so we have to use the Bismut-Weitzenb\"ock formula in combination with a scaling trick in order to deduce  the invertibility of the twisted Dirac operator.

\begin{proposition}[Scaling and invertibility of twisted Dirac]\label{bismut-vanishing}
Let $Y$ be a compact spin manifold of odd dimension and $H$ an odd degree differential form on it as above
(having no degree 1 component). 
Let $g$  be a Riemannian metric of positive scalar curvature on $Y$. Then in the scaled Riemannian metric, $\dirac^{rg, \E}_{H}$ is invertible for all $r\gg 0$.
\end{proposition}

\begin{remark}
The proofs also show that the twisted Dirac operator $\dirac_H^{rg}$, viewed as a regular self-adjoint operator on the Mishchenko Hilbert module, is invertible for $r\gg0$. 
\end{remark}

\begin{proof}
Now suppose that we scale the Riemannian metric $g$ by a positive scalar $r>0$, that is $rg$. The new scalar 
curvature $R^{rg} = R^g/r^2$ and since the induced Riemannian metric on the cotangent bundle $T^*Y$ gets scaled by 
$1/r$, it follows that the induced Riemannian metric on $\Lambda^{2k+1} T^*Y$ gets scaled by $1/r^{2k+1}$. 
Consider first the case when $k=1$, that is when $H$ is a degree 3 differential form on $Y$. Then by the scaling in
section 1 and
Theorem \ref{thm:weit}, one has
$$
\left(\dirac^{rg,\E}_{H}\right)^2 = \Delta_H^{rg,\E}  +  \frac{R^g}{4r^2} - 2 \frac{\vert H\vert^2_g}{r^3} + 
\frac{c(dH)}{r^4}.
$$
Since $R^g>0$ by hypothesis and since $Y$ is compact, therefore $\inf(R^g)>0$. 
%Then for all $r>\displaystyle \frac{8 \vert H\vert^2_g}{\inf(R^g)} >0$,
Then for all $r\gg 0$,
we see that the operator $\left(\dirac^{\E,rg}_{H}\right)^2>0$ is invertible.

In the general case, when $H$ is an odd degree differential form on $Y$, we use Theorem \ref{thm:bismut} instead. Upon scaling the Riemannian metric to $rg$, where $r>0$,
the local orthonormal frame $\{e_j: j=1,2, \ldots, n\}$ for $g$ gets scaled to the new  local orthonormal frame
$\left\{\displaystyle\frac{e_j}{\sqrt{r}}: j=1,2, \ldots, n\right\}$ in the scaled Riemannian metric $rg$. Then the expression given by equation \eqref{eqn:weit}
in the scaled Riemannian metric $rg$ becomes,
$$
 %c(H)^2 +  \sum_{j=1}^n c(\iota_{e_j} H)^2 - c(H^2) = 
 \sum_{j_1<j_2\ldots<j_k, \, k\ge 2}
 (-1)^{\frac{k(k+1)}{2}} (1-k) r^{-k}c((\iota_{e_{j_1}}\iota_{e_{j_2}}\ldots \iota_{e_{j_k}}H)^2).
$$
Since $H$ consists only of odd degree forms, the sum vanishes when $k$ is an even number. So in our case it reduces to 
$$
% c(H)^2 +  \sum_{j=1}^n c(\iota_{e_j} H)^2 - c(H^2) 
\sum_{j_1<j_2\ldots<j_k, \, k\ge 3, \, k \, odd}
 (-1)^{\frac{k(k+1)}{2}} (1-k) r^{-k}c((\iota_{e_{j_1}}\iota_{e_{j_2}}\ldots \iota_{e_{j_k}}H)^2).
$$
Then the scaled version the Lichnerowicz-Weitzenb\"ock formula of Theorem \ref{thm:bismut} becomes
\begin{align*}
\left(\dirac^{rg,\E}_{H}\right)^2 & = \Delta_H^{rg,\E}  + \frac{R^g}{4r^2}  +
\sum_{j_1<j_2\ldots<j_k, \, k\ge 3, \, k \, odd}
 (-1)^{\frac{k(k+1)}{2}} (1-k) r^{-k}c((\iota_{e_{j_1}}\iota_{e_{j_2}}\ldots \iota_{e_{j_k}}H)^2)\\
& +\sum_{j>0} r^{-{2j+2}} c(dH_{2j+1}) + c(H_r^2),
\end{align*}
where we use the notation of Proposition \ref{prop:conf}.
Since $R^g>0$ by hypothesis and since $Y$ is compact, therefore $\inf(R^g)>0$. Also all the other 
terms scale by a power of $r^{-3}$ or higher.
Therefore for all $r\gg 0$
%, one has, 
%$$
%\inf(R^g) > \sum_{j_1<j_2\ldots<j_k, \, k\ge 3,  \, k \, odd}
% (k-1) r^{-k+2}||c((\iota_{e_{j_1}}\iota_{e_{j_2}}\ldots \iota_{e_{j_k}}H)^2)||,
%$$
%and therefore 
the operator $\left(\dirac^{\E,rg}_{H}\right)^2>0$ is invertible.

\end{proof}

%%%%%%%%%%%%%%%%
\section{Review of the Higson-Roe exact sequence}
%%%%%%%%%%%%%%%%
We first recall the construction of the Higson-Roe exact sequence and its relation with the Baum-Connes assembly map. 
Much of this material is detailed in \cite{HigsonRoe}, so we shall be brief.

Recall the Mishchenko flat line bundle $\maL_Y:= {\tilde Y}\times_\Gamma C^*\Gamma \to Y$ over $Y$. The typical fiber of $\maL_Y$ is the $C^*\Gamma$-Hilbert module $C^*\Gamma$ itself.  We then define the Hilbert module $\maH_Y$ of $L^2$-sections of the bundle $\maL_Y\otimes S\to Y$. The Hilbert space of $L^2$-spinors $L^2(Y, S)$ is denoted $\fH$. If a bounded operator $T$ on $\fH$ is the reduction mod $\Gamma$ of a $\Gamma$-invariant operator $\tilde T$ on $C^\infty (\tilde Y, \tilde S)$ then using $\tilde T \otimes I$ we see that the operator $T$ lifts to a bounded operator on the Hilbert module $\maH_Y$. 

\begin{definition}
\begin{enumerate}
\item Define a $C^*$-algebra $D^*_\fH(Y)$ as being the space of bounded operators $T$ on $\fH$ such that for any $f\in C(Y)$, the commutator $[T, f]$ is a compact operator. 
\item Define the $C^*$-algebra $\maQ_\fH^*(Y)$ as the quotient of $D^*_\fH(Y)$ by the ideal $\maK(\fH)$ of compact operators.
\item We define $D^*_\Gamma (Y)$ as the space of bounded operators on $\maH_Y$ which are lifts of operators from $D^*_\fH(Y)$.
\end{enumerate}
\end{definition}
It is clear that the lift of such a $T$ is a Hilbert module compact operator on $\maH_Y$ if and only if $T$ is compact on $\fH$. Moreover, the quotient  of 
$D^*_\Gamma (Y)$ by the ideal of Hilbert module compact operators $\maK(\maH_Y)$, which can be denoted similarly $\maQ_\Gamma^*(Y)$, is actually isomorphic to the $C^*$-algebra $\maQ_\fH^*$. More generally, for any finite sub complex $X$ of the classifying space $B\Gamma$, we may defined  the Mishchenko flat line bundle $\maL_X:= {\tilde X}\times_\Gamma C^*\Gamma \to X$ over $X$.  If $X$ contains $f(Y)$, then the algebra $C(X)$ represents in the Hilbert  space $\fH$  and we can still define the Hilbert module $\maH_X$ as the completion of the space
$$
C (X, \maL) \otimes _{C(X)\times C^*\Gamma} (\fH\otimes C^*\Gamma).
$$
Then the $C^*$-algebras $D^*_\fH(X)$, $\maQ^*_\fH(X)$ and $D^*_\Gamma (X)$ are also defined similarly, see \cite{HigsonRoe}. We thus have the exact sequence:
\begin{equation}\label{ExactSequenceX}
0 \to \maK (\maH_X)  \rightarrow D^*_\Gamma (X) \rightarrow D^*_\Gamma X /\maK (\maH_\Gamma)\simeq \maQ_\fH(X)  \to 0
\end{equation}
The $K$-theory 
$K_*(D^*_\Gamma (X) /\maK (\maH_X))\simeq K_* (\maQ_\fH^*(X))$ is easily seen via Pashke duality to be isomorphic to $K_{*+1} (X)$ \cite{Pashke}, the only important point here is that $\fH$ be an ample representation of $C(X)$. On the other hand the algebra $\maK (\maH_X)$ is Morita equivalent to 
$C^*\Gamma$ and hence has the same $K$-theory.  Hence the periodic 6-term exact sequence induced by \eqref{ExactSequenceX} in $K$-theory thus gives in particular the following exact sequence
\begin{equation}\label{ExactSequenceXK}
\cdots \rightarrow K_0(X) \stackrel{\mu_{X}}{\rightarrow} K_0(C^*\Gamma) \rightarrow K_0( D^*_\Gamma (X) ) \rightarrow K_1 (X) \stackrel{\mu_{X}}{\rightarrow} K_1(C^*\Gamma) \rightarrow \cdots
\end{equation}
The connecting morphisms $\mu_X$ are precisely the Baum-Connes maps in $K_0$ and $K_1$, see again \cite{HigsonRoe}.  Passing to the direct limit over such finite sub complexes $X$ of $B\Gamma$, one gets the Higson-Roe exact sequence
\begin{equation}\label{ExactSequenceK}
\cdots \rightarrow K_0(B\Gamma) \stackrel{\mu}{\rightarrow} K_0(C^*\Gamma) \rightarrow \maS_1 (\Gamma) \rightarrow K_1 (B\Gamma) \stackrel{\mu}{\rightarrow} K_1(C^*\Gamma) \rightarrow \cdots
\end{equation}
where we have denoted by $\maS_1 (\Gamma)$ the group 
$$
\maS_1 (\Gamma) := \lim_X K_0 (D^*_\Gamma (X)).
$$

Let $\sigma_\E:\Gamma\to U(\rank\E)$ denote the unitary representation determined by the flat bundle $\E$ over $X$,
and $\sigma_{\rank\E}:\Gamma\to U(\rank\E)$ denote the trivial representation. Then the induced maps 
$${\sigma_\E}_*, {\sigma_{\rank\E}}_*
: K_0(C^*\Gamma)\to K_0(M(\rank\E, \C))
\cong\Z$$
where $M(\rank\E, \C)$ denotes matrices over the complex numbers of rank equal to $\rank\E$. Then define the trace
\begin{equation}\label{trE}
\Tr_\E :  K_0(C^*\Gamma)\longrightarrow \Z
\end{equation}
as the difference $\Tr_\E = {\sigma_\E}_* - {\sigma_{\rank\E}}_*$.

Our next goal is to recall the definition of the relative trace map $\Tr_\E : \maS_1(\Gamma) \to \R$ having the key property that the following diagram commutes,
\begin{equation}\label{DiagramTraces}
\begin{CD}
K_0(B\Gamma) @>\mu>> K_0(C^*\Gamma) @>>>   \maS_1(\Gamma) @>>> K_1(B\Gamma) @>\mu>> K_1(C^*\Gamma) \\
        @VVV          @V\Tr_\E VV   @V\Tr_\E VV   @V\Ind_\E VV              @VVV \\
0   @>>> \Z @>>>   \R @>>> \R/\Z@>>> 0
\end{CD}
\end{equation}
In order to define $\Tr_\E : \maS_1(\Gamma) \to \R$, we need to
define a structure group $\maS_1(\E)$ and a natural transformation 
$\maS_1(\Gamma)\to \maS_1(\E)$ determined by the representations $\sigma_\E, \sigma_{\rank\E}$.
Then we recall the definition of the geometric structure 
group $\maS_1^{geom}(\E)$, defined analogously to geometric K-homology and a natural 
isomorphism $\maS_1(\E) \cong \maS_1^{geom}(\E)$. Using the twisted rho invariant, we also have 
a morphism $\rho: \maS_1^{geom}(\E) \to \R$. Finally, the relative trace map $\Tr_\E : \maS_1(\Gamma) \to \R$
is defined as the composition
$$
\maS_1(\Gamma)\to \maS_1(\E) \cong \maS_1^{geom}(\E) \stackrel{\rho}{\longrightarrow} \R.
$$

Let $\fH_\E = C(X, \E)\otimes_{C(X)} \fH$ and $\fH_{\rank\E} = C(X, \rank\E)\otimes_{C(X)} \fH$. Denote by 
$D^*_{\fH_\E, \fH_{\rank\E}} (X)$ the $C^*$-algebra
$$
\left\{(T_1, T_2) \in  D^*_{\fH_\E} (X) \oplus D^*_{\fH_{\rank\E}} (X) \,\, \Big|\,\,\, T_1, T_2\,\, \text{lift to some common}\, \, T \in D^*_{\fH}(X)    \right\}
$$
Then define the structure group
$$
\maS_1(\E) := \lim_X K_0 (D^*_{\fH_\E, \fH_{\rank\E}} (X)).
$$
Viewing $\fH_\E = \fH \otimes_{C^*\Gamma} \C^{\rank\E}_{\sigma_\E}$ and 
$\fH_{\rank\E} =  \fH \otimes_{C^*\Gamma} \C^{\rank\E}_{\sigma_{\rank\E}}$, 
the map $T\mapsto (T\otimes I, T\otimes~I)$ gives a natural morphism of $C^*$-algebras $
D^*_\Gamma (X) \to  D^*_{\fH_\E, \fH_{\rank\E}} (X)$ and hence a natural transformation 
$\maS_1(\Gamma) \to \maS_1(\E)$.

\begin{definition} An (odd) geometric $\E$-cycle is a quintuple $(Y, S, f, D, n)$ where
\begin{enumerate}
\item $(Y, S, f)$ is a geometric K-cycle for $B\Gamma$. 
\item $D$ is a specific choice of Dirac operator for $(Y, S, f)$.
\item $n$ is an integer.
\end{enumerate}
\end{definition}

\begin{definition} A geometric $\E$-cycle $(Y, S, f, D, n)$ is a boundary if there is a Dirac operator $Q$ on a compact manifold $W$ whose boundary is the operator $D$ on $Y$, if the map $f$ and the bundle $\E$ extends to $W$, and if
$\Ind(Q_\E, D_\E) - \Ind(Q_{\rank\E}, D_{\rank\E}) = n$. 
\end{definition}
Then one has
\begin{lemma} If a geometric $\E$-cycle $(Y, S, f, D, n)$ is a boundary, then
$\rho(D, f, \E) + n = 0$.
\end{lemma}

\begin{definition} We shall call the quantity $\rho(D, f, \E) + n$ the twisted rho invariant of the geometric cycle 
$(Y, S, f, D, n)$.
There is an obvious notion of disjoint union of geometric $(\E)$-cycles.
There is also a subtle notion of the negative of a cycle,
$-(Y,S,f,D,n)=(Y,-S,f,-D, d_\E - d_{\rank\E} -n)$, where $d_\E = \dim \ker(D_\E)$ and $d_{\rank\E} = \dim \ker(D_{\rank\E})$.
\end{definition}

\begin{definition}  Two geometric $(\E)$-cycles are {\em bordant} if the disjoint union of one with the negative of the other is a boundary.
\end{definition}

\begin{definition}
$ \maS_1^{geom}(\E) $ is the set of equivalence classes of geometric $(\E)$-cycles under the equivalence relation generated by:
\begin{enumerate}
\item {\em Direct sum/disjoint union}. The cycle $(M, S' \oplus S'' , f, D' \oplus D'' , n)$ is equivalent to 
$(M\coprod M,S' \coprod S'',f\coprod f,D' \oplus D'',n)$.
\item {\em Bordism}.	If	$(M', S', f', D', n')$	and	$(M'', S'', f'', D'', n'')$	are	bordant,	then they are equivalent.
\item  {\em Vector bundle Modification}. If $(\hat Y, \hat S, \hat f)$ is a vector bundle modification of 
$(Y, S, f)$ and if
$\hat D = D \otimes \varepsilon + I \otimes D_\theta$ is the specific Dirac operator that is adapted to the geometry,
 then $(\hat Y, \hat S, \hat f, \hat D, n)$ and $(Y, S, f, D, n)$ are equivalent.
\end{enumerate}
The set $ \maS_1^{geom}(\E) $ is an abelian group with 
addition given by disjoint union.
\end{definition}

Then one has
\begin{proposition} The twisted rho invariant of a geometric $(\E)$-cycle depends only on the class that the cycle determines in $\maS_1^{geom}(\E) $.
\end{proposition}
\begin{definition} We define the group homomorphism $\rho:\maS_1^{geom}(\E)\longrightarrow \R$
by the formula
$$
\rho : (Y, S, f, D, n)\mapsto \rho(D, f, \E ) + n.
$$
\end{definition}

Finally, one has the following
\begin{proposition} The group homomorphism
$$
\maS_1^{geom}(\E)\longrightarrow \maS_1(\E)$$
that associates to any geometric $(\E)$-cycle its analytic structure class is an isomorphism.
\end{proposition}

%%%%%%%%%%%%%%%%
\section{Spectral sections and twisted eta invariant}\label{sect:spectral}
%%%%%%%%%%%%%%%%

%%%%%%%%%%%%%%%%
\subsection{Review of spectral sections}\label{SpectralSections}
%%%%%%%%%%%%%%%%

We gather here some results about noncommutative spectral sections that we use in the sequel. Our main references are \cite{MelrosePiazza, LeichtnamLottPiazza, LeichtnamPiazza, PiazzaSchick, Wu}. Assume that $\mathfrak E$ is a Hilbert module over a unital $C^*$-algebra $A$ and denote as usual by $B_A({\mathfrak E})$ and $\maK_A ({\mathfrak E})$ the space of adjointable operators on ${\mathfrak E}$ and adjointable $A$-compact operators on ${\mathfrak E}$, respectively. We assume that $\mathfrak E$ is a full Hilbert module, i.e. that the subspace $<{\mathfrak E}, {\mathfrak E}>$ is dense in $A$. Let $D$ be a densely defined regular self-adjoint operator on ${\mathfrak E}$, then we have the well defined bounded  continuous  functional calculus of $D$, see for instance \cite{Lance, BJ}. Hence if $\chi_{\geq 0}$ is the characteristic the APS projection $\pi_+(D)=\chi_{\geq 0} (D)$ is well defined in $B_A({\mathfrak E})$. A spectral section for $D$ is roughly speaking a projection in $B_A({\mathfrak E})$ which differs from the operator $\frac{1}{2} \left(I + (I+D^2)^{-1/2} D\right)$ by an $A$-compact operator. Let us give a more precise definition now. A smooth cut is a smooth function $\chi: \R\to [0,1]$ such that there exist $s_1<s_2\in \R$ with 
$$
\chi (t) = 0 \text{ for } t\leq s_1 \text{ and } \chi (t) = 1\text{ for } t\geq s_2.
$$

\begin{definition}
Let $({\mathfrak E}, D)$ be as above. A spectral section for $D$ is a projection $\maP=\maP^*=\maP^2$ in $B_A({\mathfrak E})$ such that there exist smooth cuts $\chi_1, \chi_2$ with $\chi_1\chi_2 = \chi_1$ and
$$
\chi_1(D) ({\mathfrak E}) \subset \maP ({\mathfrak E}) \subset \chi_2 (D) ({\mathfrak E}).
$$
\end{definition}

If a spectral section $\maP$ for $D$ is given then for any continuous function $f:\R\to \R$ such that $\lim_{-\infty} f = 0$ and $\lim_{+\infty} f = 1$,  the operator $\maP-f (D)$ is $A$-compact. 
Assume that $D$ has $A$-compact resolvent, then the Cayley transform operator $(D-iI)(D+iI)^{-1}$ of $D$ differs from the identity operator $I$ by an $A$-compact operator, and hence defines an analytic index class $\ind_a(D)$ in the $K$-theory group $K_1 (\maK_A({\mathfrak E}))\simeq K_1 (A)$. It is easy to show that the existence of a spectral section for $D$ implies the vanishing of the  index class $\ind_a(D)$. Conversely, when this index class vanishes, one can construct a spectral section for $D$ by using an appropriate  $A$-compact perturbation. More precisely, we have the following result  proved in \cite{LeichtnamPiazza}:

\begin{theorem}\cite{LeichtnamPiazza}
With the above notations, the index class $\ind_a(D)$ vanishes in $K_1(A)$ if and only if $D$ admits spectral sections.
\end{theorem}

The difference between two spectral sections for $D$ is measured by the $K$-theory class it defines in $K_0(\maK_A({\mathfrak E}))\simeq K_0(A)$. 

The main example for the present paper will be the covering situation $\Gamma\to {\tilde Y} \to Y$ of the smooth closed odd-dimensional spin manifold $Y$. We then denote by $\Lambda=A$ the group $C^*$-algebra of $\Gamma$ and consider the Mishchenko bundle $\maL_Y:= {\tilde Y}\times_\Gamma C^*\Gamma \to Y$ over $Y$. Then the Dirac operator $D=\dirac$, acting on the full Hilbert module $\maH_Y={\mathfrak E}$ of $L^2$-sections of the Mishchenko bundle with coefficients in the spin bundle $S$, is a regular self-adjoint operator which has compact resolvent. In fact, one can use the Mishchenko pseudodifferential calculus over $Y$ (corresponding to the bundle $\maL_Y\otimes S$ of projective modules over the $C^*$-algebra $\Lambda$) and construct in this case a smoothing operator $C$ such that $\dirac+C$ is invertible in this calculus. Then a spectral section for $\dirac$ is simply obtained as the APS projection of this perturbed Dirac operator. It is then clear that such spectral section exists as a zero-th order pseudodifferential operator in the Mishchenko calculus. 
Hence we get in this case the following equivalence, see  \cite{LeichtnamPiazzaJFA}
\begin{align*}
\ind_a (\dirac) & =0 \in K_1(\Lambda)  \Longleftrightarrow \\
& \dirac \text{ admits a spectral section which is a $0$-th order pseudodifferential operator}. 
\end{align*}

If $\sigma:\Gamma \to U({\mathfrak K})$ is a unitary representation in some Hilbert space $\mathfrak K$, then we may use the composition construction $T\mapsto T_\sigma:=T\otimes I$ to transform operators on the Hilbert module $\maH_Y$ into operators on the Hilbert space $\maH_Y\otimes_\sigma {\mathfrak K}$, see for instance \cite{Lance}. Then we get  the self-adjoint Dirac operator $\dirac_\sigma$ and any spectral section $\maP$ for $\dirac$ gives rise to a projection $\maP_\sigma$. When $\mathfrak K$ is finite dimensional, the self-adjoint Dirac operator $\dirac_\sigma$ is just the Dirac operator $\dirac^{\maE_\sigma}$ twisted by the flat vector bundle $\maE_\sigma:={\tilde Y}\times_\Gamma {\mathfrak K}$ over $Y$ associated with the representation $\sigma$ and it is hence an elliptic operator with discrete spectrum. Then, any spectral section $\maP$ for $\dirac$ gives rise to a spectral section $\maP^{\maE_\sigma}:=\maP_{\sigma}$ for $\dirac^{\maE_\sigma}$ which is nothing but an appropriate smooth spectral cut.

When $\sigma=\lambda$ is the regular left representation of $\Gamma$ in the Hilbert space ${\mathfrak K}=\ell^2\Gamma$, it is easy to identify the Hilbert space $\maH_Y\otimes_\lambda \ell^2\Gamma$ with the Hilbert space $L^2({\tilde Y}, {\tilde S})$ of $L^2$ spinors on the cover $\tilde Y$ \cite{BenameurPiazza}. The operator $\dirac_\sigma$ is then the $\Gamma$-invariant Dirac operator over $\tilde Y$. In this case,  any spectral section $\maP$ gives rise to a $\Gamma$-invariant projection $\maP_\lambda$ but it is not obvious at all  that such spectral section is a pseudodifferential operator and hence that it belongs to the (maximal) Higson-Roe algebra. When one works with the reduced Higson-Roe algebra, then the following is implicitly proved in \cite{PiazzaSchick}

\begin{proposition}\cite{PiazzaSchick}
Assume that the index class $\ind_a(\dirac)$ vanishes in $K_1(C_{\rm {red}}^*\Gamma)$. Then there exists a spectral section $\maP$ for $\dirac$ which belongs to the reduced Higson-Roe $C^*$-algebra.
\end{proposition}

\begin{proof}
The first observation is that the reduced Higson-Roe algebra $D^*_{\rm{red}, \Gamma} (Y)$ is composed of operators which are by definition lifts of pseudo local finite propagation $\Gamma$-invariant operators on the cover $\Gamma\to {\tilde Y}\to Y$ and it can be defined using these $\Gamma$-invariant pseudo local operators on the cover ${\tilde Y}$ with finite propagation.  Denote by ${\tilde\dirac}$ the $\Gamma$-invariant Dirac operator acting on the $L^2$-sections of ${\tilde S}$ over ${\tilde Y}$. When the index class $\ind_a(\dirac)$ vanishes, one can construct a $\Gamma$-invariant self-adjoint bounded perturbation ${\tilde \dirac} + {\tilde A}$ of ${\tilde\dirac}$ which is $L^2$-invertible and composed of pseudo differential operators (see for instance \cite{LeichtnamLottPiazza}). The classical theory of pseudo differential operators then shows that the APS projection (associated the nonnegative $L^2$-spectrum) of ${\tilde \dirac} + {\tilde A}$ acting over ${\tilde Y}$ is  a ($\Gamma$-invariant) zeroth order pseudo differential operator which has finite propagation. Lifting this projection as an adjointable operator on the Mishchenko Hilbert module, we immediately deduce a projection in the reduced Higson-Roe $C^*$-algebra   $D^*_{\rm{red}, \Gamma} (Y)$ which is a particular choice of spectral section for the Dirac operator. 
\end{proof}

\begin{remark}
When the group $\Gamma$ is $K$-amenable, a simple five-lemma argument shows that under the assumptions of the previous proposition, one insures the existence of a $K$-theory class of a spectral section for the maximal Higson-Roe algebra $D^*_{\Gamma} (Y)$
\end{remark}

{\bf{In the sequel and unless otherwise specified,  we shall always assume for simplicity that our spectral sections belong to the maximal Higson-Roe algebra and hence define  classes in the inductive limit group $\maS_1\Gamma$. }}

%%%%%%%%%%%%%%%%
\subsection{Structures associated with the twisted Dirac operator}
%%%%%%%%%%%%%%%%

We assume as before that $Y$ is a closed odd-dimensional spin manifold with a universal map $f: Y\to B\Gamma$ which is uniquely defined up to homotopy. We denote by $S\to Y$ the chosen spin bundle on $Y$ so that the triple $(Y,f,S)$ defines a Baum-Douglas $K$-homology class $[Y,f,S]$ in $K_1(B\Gamma)$ \cite{BaumDouglas}. This $K$-homology class can be described in the Kasparov picture of $K$-homology by the Dirac operator $\dirac$ on $Y$. Moreover, for any complex odd-degree differential  form $H$ on $Y$ which satisfies the reality condition in equation \eqref{Reality} 
(discussed in \cite{BM-JGP}),  the twisted Dirac operator $\dirac_H$ also represents the same Kasparov $K$-homology class $[\dirac_H]=[\dirac]$  in $K_1 (Y)$, and hence also the same pushforward class $f_*[\dirac_H]=f_*[\dirac] \in K_1(B\Gamma)$. 

The (analytic) index class  $\ind_a (\dirac_H)$ of the twisted Dirac operator $\dirac_H$ is by definition  the class represented by the Cayley transform $(\dirac_H - i I)(\dirac_H+iI)^{-1}$ of $\dirac_H$ in $K_1(\maK (\maH_Y))\simeq K_1 (C^*\Gamma)$. Consider now the linear  path $(\dirac_{v H})_{0\leq v \leq 1}$ of Dirac operators acting on the Mishchenko Hilbert module. The index class of this path belongs to $K_1(C([0,1]), C^*\Gamma)\simeq K_1(C^*\Gamma)$. Recall that this index class vanishes if and only if there exists  a spectral section $\maP=(\maP_v)_{0\leq v\leq 1}$ for the global path $(\dirac_{v H})_{0\leq v \leq 1}$. The operator $\maP_v$ is a zero-th order pseudo differential projection in the Mishchenko calculus, see \cite{LeichtnamPiazza}. Notice that the index class coincides via the isomorphism $K_1(C([0,1]), C^*\Gamma)\simeq K_1(C^*\Gamma)$ with the index class of $\dirac$, which in turn is obviously equal to the index class of $\dirac_H$. 
Hence, when this common index class vanishes, there also exist spectral sections $\maQ$ and $\maQ_H$ for the Dirac operators $\dirac$ and $\dirac_H$ respectively. But $\maP_0$ is then a second spectral section for $\dirac$ while $\maP_1$ is a second spectral section for $\dirac_H$.  Now by definition of spectral sections, the difference $\maP_1 - \maQ_H$ and the difference $\maP_0-\maQ$ define $K_0$-classes of the $C^*$-algebra of compact operators  on the Mishchenko Hilbert module and so of $C^*\Gamma$. 

\begin{definition} \cite{Wu}
Assume that the index class $\Ind_a(\dirac)=0$ in $K_1(C^*\Gamma)$. Then the higher spectral flow $\SF ((\dirac_{v H})_{0\leq v \leq 1}; \maQ, \maQ_H)$ of the path $(\dirac_{v H})_{0\leq v \leq 1}$, with respect to the given spectral sections $\maQ$ and $\maQ_H$, is the $K$-theory class
$$
\SF ((\dirac_{v H})_{0\leq v \leq 1}; \maQ, \maQ_H) := [\maQ_H - \maP_1] - [\maQ - \maP_0] \quad \in K_0 (C^*\Gamma).
$$
\end{definition}

The notation $\SF ((\dirac_{v H})_{0\leq v \leq 1}; \maQ, \maQ_H)$ is consistent since the higher spectral flow does not depend on the choice of global spectral section $\maP=(\maP_v)_{0\leq v\leq 1}$, \cite{LeichtnamPiazza, DaiZhang, Wu}. In the literature, the higher spectral flow usually refers to the Connes-Chern character of the above $K$-theory higher spectral flow and so lives in cyclic homology \cite{BC}. In the present paper though, we shall restrict to traces associated with finite dimensional representations of the group $\Gamma$.

As recalled above in Subsection \ref{SpectralSections}, when $\ind_a (\dirac)=0$ in $K_1(C^*\Gamma)$, there exists a spectral section $\maQ$ for $\dirac$ which moreover represents a class in $\maS_1\Gamma$. In the sequel, it is understood that  all spectral sections associated with $\dirac$ or $\dirac_H$ belong to $D^*_\Gamma (Y)$ and same for the spectral sections associated with the path $(\dirac_{vH})_{0\leq v\leq 1}$ which then belongs to $D^*_\Gamma (Y\times [0,1])$. 

\begin{proposition}\label{SS} Assume that $\ind_a(\dirac)=0$ in $K_1(C^*\Gamma)$. Then:
\begin{itemize}
\item  Any spectral section $\maQ$ for $\dirac$ represents a class in $\maS_1(\Gamma)$ which is a pre image of $f_*[\dirac]\in K_1(B\Gamma)$ in the Higson-Roe exact sequence \eqref{ExactSequenceK}.
\item Given spectral sections $\maQ$ and $\maQ_H$ for $\dirac$ and $\dirac_H$ respectively, the class  $[\maQ_H] - [\maQ]$ in $\maS_1(\Gamma)$ is the image of the class $\SF ((\dirac_{v H})_{0\leq v \leq 1}; \maQ, \maQ_H) \in K_0(C^*\Gamma)$ in the Higson-Roe exact sequence \eqref{ExactSequenceK}.
\end{itemize}
\end{proposition}

\begin{remark}
In \cite{PiazzaSchick}, a class $[\maQ]\in \maS_1(\Gamma)$ as above is called a rho class associated with a given perturbation of the Dirac operator. 
\end{remark}

\begin{proof}
As explained in Subsection \ref{SpectralSections}, the spectral section $\maQ$ defines  an element of the inductive limit group $\maS_1(\Gamma)$.  To describe a pre image of the $K$-homology class $[\dirac]\in K_1(Y)$ living in $K_0(D^*_\Gamma (Y))$, we denote by $\varphi$ the smooth function defined by
$$
\varphi (t):=\frac{1}{\pi} \arctan (t) + \frac{1}{2}.
$$
We then define the operator ${\tilde P}:=\varphi(\dirac)$. By \cite{HigsonRoe1} (see also \cite{PiazzaSchick}[Proposition 2.8]), we know that ${\tilde P} \in D^*_\Gamma Y$. Consider the affine path $\gamma$ which joins ${\tilde P}$ to $\maQ$ in $D^*_\Gamma Y$.  An easy uniform convergence argument  shows that for any $t$, the operator $e^{2i\pi \gamma (t)}$ differs from the identity by a compact operator. Therefore, using \cite{HigsonRoe}[Lemma 7.4], we deduce that the  path $\gamma$ defines a class in $\maS_1(\Gamma)$ whose image is the class of the Cayley transform of $\dirac$. The path $\gamma$ being homotopic to the constant path equal to $\maQ$, we see that its class also coincides with the class defined by the projection $\maQ$. This eventually shows that the projection $\maQ$ represents a class in the inductive limit  $\maS_1\Gamma$ which is sent to the class $f_*[\dirac]$ in $K_1(B\Gamma)$ in the Higson-Roe exact sequence.  

Choose again a global spectral section $\maP=(\maP_v)_{0\leq v\leq 1}$ for the family $(\dirac_{v H})_{0\leq v \leq 1}$, which again lives in $D^*_\Gamma (Y\times [0,1])$. Then the class $\SF ((\dirac_{v H})_{0\leq v \leq 1}; \maQ, \maQ_H) \in K_0(C^*\Gamma)$ is
$
[\maQ_H-\maP_1] - [\maQ - \maP_0].$
This class is sent to the class in $K_0(D^*_\Gamma(Y))$ given by
$$
([\maQ_H] - [\maQ] ) - ([\maP_1] - [\maP_0]).
$$
But the projections $\maP_0$ and $\maP_1$ define the same class in $K_0(D^*_\Gamma (Y))$, hence the spectral flow $\SF ((\dirac_{v H})_{0\leq v \leq 1}; \maQ, \maQ_H) $ is sent to the class $[\maQ_H]- [\maQ]$ in $\maS_1(\Gamma)$ as claimed.
\end{proof}

Recall from equation \eqref{trE}, the definition of the additive maps $\Tr_\maE$ on $K_0(C^*\Gamma)$ and on $\maS_1\Gamma$. In \cite{HigsonRoe}, these maps were denoted respectively $\Tr_{\sigma_1} - \Tr_{\sigma_2}$ and $\Tr_{\sigma_1, \sigma_2}$ where $\sigma_1$ is the monodromy representation associated with the flat bundle $\maE$ while $\sigma_2$ is the trivial representation corresponding to $\Rank(\maE)$.

\begin{corollary}\label{nH}\ We keep the notations of the previous proposition. Denote by $\maQ_H^\maE$ and $\maQ_H^{\Rank \maE}$ the spectral sections obtained from $\maQ_H$ through the representation defining the flat bundle $\maE$ and the trivial bundle ${\Rank \maE}$ respectively. Define in the same way $\maQ^\maE$ and $\maQ^{\Rank\maE}$ out of the spectral section $\maQ$. Then the image $n (H, \maE, \maQ, \maQ_H)$ of the higher spectral flow $\SF ((\dirac_{v H})_{0\leq v \leq 1}; \maQ, \maQ_H) $ under the trace map $\Tr_\maE$ is the difference of spectral flows
$$
\sf ((\dirac^\maE_{v H})_{0\leq v \leq 1}; \maQ^\maE, \maQ^\maE_H) - \sf ((\dirac^{\Rank\maE}_{v H})_{0\leq v \leq 1}; \maQ^{\Rank\maE}, \maQ^{\Rank\maE}_H).
$$
More precisely, for any choice of a global spectral section $\maP$ we have 
\begin{multline*}
n (H, \maE, \maQ, \maQ_H) 
=\left[ \Tr (\maQ^\maE_H - \maP^\maE_1) - \Tr (\maQ^\maE - \maP^\maE_0)\right] - \left[  \Tr (\maQ^{\Rank\maE}_H - \maP^{\Rank\maE}_1) - \Tr (\maQ^{\Rank\maE} - \maP^{\Rank\maE}_0)\right] \\ = \sf ((\dirac^\maE_{v H})_{0\leq v \leq 1}; \maQ^\maE, \maQ^\maE_H) - \sf ((\dirac^{\Rank\maE}_{v H})_{0\leq v \leq 1}; \maQ^{\Rank\maE}, \maQ^{\Rank\maE}_H)
\end{multline*}
\end{corollary}

\begin{proof}
We use the second item of Proposition \ref{SS} and the commutativity of the following diagram (proved in \cite{HigsonRoe}):

\hspace{0.25in}
\begin{picture}(415,80)
%\put(10,60){$\oH^*_{\Delta, 2}(\wL')$} 
%\put(20,50){$\vector(0,-1){20}$}
%\put(10,10){$\oH^*_{(2)}(\wL')$} 
%\put(25,35){$W$}

%\put(70,70){$ p_{f,\tri}^*$}
%\put(55,64){\vector(1,0){40}}
%\put(55,13){\vector(1,0){40}}
%\put(70,19){$p_f^*$}

\put(115,60){$K_0(C^*\Gamma)$}
\put(135,50){ $\vector(0,-1){20}$}
\put(135,10){$\Z$}
\put(145,35){$\Tr_\maE$}

%\put(180,70){$[\beta]\cup $}
\put(170,64){\vector(1,0){40}}
\put(170,13){\vector(1,0){40}}
%\put(180,19){$[\omega] \wedge$}

\put(220,60){$\maS_1\Gamma$}
\put(225,50){ $\vector(0,-1){20}$}
\put(225,10){$\R$}
\put(235,35){$\Tr_\maE$}

%\put(295,70){$\cap [B^k]$}
%\put(290,64){\vector(1,0){40}}
%\put(290,13){\vector(1,0){40}}
%\put(300,19){$\pi_{1,*}$}
%
%\put(340,60){$\oH^*_{\Delta, 2}(\wL)$}
%\put(340,50){ $\vector(0,-1){20}$}
%\put(340,10){$\oH^*_{(2)}(\wL).$}
%\put(350,35){$W$}
\end{picture}\\
to deduce that
\begin{multline*}
\Tr_\maE (\SF ((\dirac_{v H})_{0\leq v \leq 1}; \maQ, \maQ_H) ) =\left[ \Tr (\maQ^\maE_H - \maP^\maE_1) - \Tr (\maQ^\maE - \maP^\maE_0)\right]\\ - \left[  \Tr (\maQ^{\Rank\maE}_H - \maP^{\Rank\maE}_1) - \Tr (\maQ^{\Rank\maE} - \maP^{\Rank\maE}_0)\right]
\end{multline*}
But observe that
$$
\Tr_\maE (\SF ((\dirac_{v H})_{0\leq v \leq 1}; \maQ, \maQ_H) ) =  \sf ((\dirac^\maE_{v H})_{0\leq v \leq 1}; \maQ^\maE, \maQ^\maE_H) - \sf ((\dirac^{\Rank\maE}_{v H})_{0\leq v \leq 1}; \maQ^{\Rank\maE}, \maQ^{\Rank\maE}_H).
$$
\end{proof}

\begin{proposition} Assume that the Baum-Connes assembly map $\mu: K_*(B\Gamma) \to K_*(C^*\Gamma)$ is an isomorphism.   Then for any odd differential form $H$ satisfying as above the reality condition \ref{Reality}, any hermitian flat bundle  $\maE$ over the closed spin manifold $Y$ and for any spectral sections $\maQ$ and $\maQ_H$ for the operators $\dirac$ and $\dirac_H$ respectively, we have:
$$
\sf ((\dirac^\maE_{v H})_{0\leq v \leq 1}; \maQ^\maE, \maQ^\maE_H) =  \Rank (\maE) \times  \sf ((\dirac_{v H})_{0\leq v \leq 1}; \maQ, \maQ_H),
$$
In particular, $\sf ((\dirac^\maE_{v H})_{0\leq v \leq 1}; \maQ^\maE, \maQ^\maE_H) $ only depends on the rank of $\maE$.
\end{proposition}

\begin{proof}
If the Baum-Connes map is an isomorphism then $\maS_1(\Gamma)$ is trivial. We have seen in the proof of Theorem \ref{SS} that the image of the class $[\maQ_H] - [\maQ]\in \maS_1(\Gamma)$ under the additive map $\Tr_\maE$ is precisely the integer $n(H, \maE, \maQ, \maQ_H)$. Therefore, this latter integer is zero, which shows that
$$
 \sf ((\dirac^\maE_{v H})_{0\leq v \leq 1}; \maQ^\maE, \maQ^\maE_H) = \sf ((\dirac^{\Rank\maE}_{v H})_{0\leq v \leq 1}; \maQ^{\Rank\maE}, \maQ^{\Rank\maE}_H)
$$
But 
$$
\sf ((\dirac^{\Rank\maE}_{v H})_{0\leq v \leq 1}; \maQ^{\Rank\maE}, \maQ^{\Rank\maE}_H) = \Rank (\maE) \times  \sf ((\dirac_{v H})_{0\leq v \leq 1}; \maQ, \maQ_H)
$$
\end{proof}

Notice that for any $0\leq v \leq 1$, the spectra of the operators $\dirac^\maE_{v H}$ and $\dirac^{\Rank\maE}_{v H}$ are discrete. hence the spectral sections $\maQ^\maE, \maQ^\maE_H, \maQ^{\Rank\maE}$ and $ \maQ^{\Rank\maE}_H$ correspond to cuts in the spectra.

\begin{remark}
If we use the spectral sections $\pi_+(\dirac^\maE):=1_{[0, +\infty)} (\dirac^\maE)$ and $\pi_+(\dirac_H^\maE):=1_{[0, +\infty)} (\dirac_H^\maE)$ for the endpoints $\dirac^\maE$ and $\dirac_H^\maE$, then the spectral flow $\sf ((\dirac^{\maE}_{v H})_{0\leq v \leq 1}; \pi_+ (\dirac^\maE), \pi_+ (\dirac_H^\maE))$ will simply be denoted by $\sf ((\dirac^{\maE}_{v H})_{0\leq v \leq 1})$. Notice that when $\dirac^\maE$ and $\dirac_H^\maE$ are invertible, this spectral flow coincides with the usual APS spectral flow (algebraic net number of eigenvalues crossing zero), however the previous proposition does not a priori hold in general with this APS spectral flow. 
\end{remark}

We are now in position to prove the following

\begin{theorem}\label{Main} We keep the above notations.
\begin{enumerate} 
\item Assume that $\ind_a(\dirac)=0$ in $K_1(C^*\Gamma)$. Then for any $\maE$ and $H$ as before, we have
$$
\rho_{\rm {spin}} (Y, \maE, H) - \rho_{\rm {spin}} (Y, \maE) = \sf ( (\dirac^\maE_{v H})_{0\leq v \leq 1}) - \sf ((\dirac^{\Rank\maE}_{v H})_{0\leq v \leq 1}). 
$$
\item In the case that the metric $g$ on $Y$  has positive scalar curvature.  Then 
$$
\rho_{spin}(Y,\E,H, r g) = \rho_{\rm {spin}} (Y, \maE, g) \text{  for all } r\gg 0.
$$ 
In particular, there exists a metric $g'$ in the conformal class of $g$  such that  for any $\lambda\geq 1 $, we have
$
\rho_{\rm {spin}} (Y, \maE,  H, \lambda  g') = \rho_{\rm {spin}} (Y, \maE, g).
$
\end{enumerate}
\end{theorem}

\begin{proof} Assume  that $\ind_a(\dirac)=0$ in $K_1(C^*\Gamma)$ and choose a global spectral section $\maP=(\maP_v)_{0\leq v\leq 1}$ for the path $(\dirac_{vH})_{0\leq v \leq 1}$ in the Higson-Roe algebra. Then for the first item  we need to prove the relation
\begin{multline*}
\rho_{\rm {spin}} (Y, \maE, H) - \rho_{\rm {spin}} (Y, \maE) =  \left[\Tr \left( \pi_+ (\dirac_H^\maE) - \maP_1^\maE\right)  -  \Tr \left( \pi_+ (\dirac^{\maE}) - \maP_0^{\maE})\right)\right] \\ - 
\left[\Tr \left( \pi_+ (\dirac_H^{\Rank\maE}) - \maP_1^{\Rank\maE}\right)  -  \Tr \left( \pi_+ (\dirac^{\Rank\maE}) - \maP_0^{\Rank\maE})\right)\right]
\end{multline*}
Choose then spectral sections as before and set
$$
e_1^H=\maQ_H^\maE - \pi_+(\dirac_H^\maE) \text{ and } e_2^H=\maQ_H^{\Rank\maE} - \pi_+(\dirac_H^{\Rank\maE}).
$$
Then the couple $(e_1^H, e_2^H)$ defines a class of in $\maS_1(\maE)$ which corresponds to the integer 
$$
n_H=\Tr \left(\maQ_H^\maE - \pi_+(\dirac_H^\maE) \right) - \Tr \left(\maQ_H^{\Rank\maE} - \pi_+(\dirac_H^{\Rank\maE})\right).
$$
The geometric class represented by the quintuple $(Y, f , S, \dirac_H, n_H)$ in $\maS_1^{geom} (\maE)$ then corresponds to the class 
$$
[\pi_+(\dirac_H^\maE), \pi_+(\dirac_H^{\Rank\maE})] + [e_1^H, e_2^H] = [\maQ_H^\maE, \maQ_H^{\Rank\maE}].
$$
in the analytic structure group $\maS_1(\maE)$.

For $H=0$, say for the untwisted Dirac operator $\dirac$, we similarly define a pair $(e_1, e_2)$ which represents a class in $\maS_1(\maE)$ corresponding to the integer
$$
n=\Tr \left(\maQ^\maE - \pi_+(\dirac^\maE) \right) - \Tr \left(\maQ^{\Rank\maE} - \pi_+(\dirac^{\Rank\maE})\right).
$$
Again the geometric class represented by the quintuple $(Y, f , S, \dirac, n)$ in $\maS_1^{geom} (\maE)$ corresponds to the class 
$[\maQ^\maE, \maQ^{\Rank\maE}]$ in the analytic structure group $\maS_1(\maE)$.

But recall that the additive map $\Tr_\maE:\maS_1(\maE)\to \R$  is precisely defined through the isomorphism with the geometric structure group and hence that
$$
\Tr_\maE [\maQ_H^\maE, \maQ_H^{\Rank\maE}] =  \rho_{\rm {spin}} (Y, \maE, H) + n_H \text{ and } \Tr_\maE [\maQ^\maE, \maQ^{\Rank\maE}] =  \rho_{\rm {spin}} (Y, \maE) + n.
$$

On the other hand, using the commutativity of the diagram \ref{DiagramTraces}, we can write
$$
\Tr_\maE (\SF ((\dirac_{v H})_{0\leq v \leq 1}; \maQ, \maQ_H))) = \Tr_\maE ([\maQ_H]- [\maQ]) 
$$
and we also have
$$
\Tr_\maE ([\maQ_H]- [\maQ]) = \Tr_\maE [\maQ_H^\maE, \maQ_H^{\Rank\maE}] - \Tr_\maE [\maQ^\maE, \maQ^{\Rank\maE}].
$$
On the other hand
$$
\Tr_\maE (\SF ((\dirac_{v H})_{0\leq v \leq 1}; \maQ, \maQ_H))) = \sf( (\dirac^\maE_{v H})_{0\leq v \leq 1}; \maQ^\maE, \maQ_H^\maE)) - \sf( (\dirac^{\Rank\maE}_{v H})_{0\leq v \leq 1}; \maQ^{\Rank\maE}, \maQ_H^{\Rank\maE})).
$$
Using  Corollary \ref{nH}, we deduce
$$
\Tr_\maE (\SF ((\dirac_{v H})_{0\leq v \leq 1}; \maQ, \maQ_H)) = n(H, \maE, \maQ, \maQ_H).
$$
Gathering all these equalities, we can compute
$$
 \rho_{\rm {spin}} (Y, \maE, H) -  \rho_{\rm {spin}} (Y, \maE) + n_H - n = n(H, \maE, \maQ, \maQ_H).
$$
But with the choice of the global spectral section $\maP=(\maP_v)_{0\leq v\leq 1}$ in the Mishchenko calculus, we easily see that
\begin{multline*}
n(H, \maE, \maQ, \maQ_H) - (n_H-n) = \Tr \left( \pi_+ (\dirac_H^\maE) - \maP_1^\maE\right)  -  \Tr \left( \pi_+ (\dirac^{\maE}) - \maP_0^{\maE})\right) \\ 
+ \Tr \left( \pi_+ (\dirac^{\Rank\maE}) - \maP_0^{\Rank\maE})\right)  - \Tr \left( \pi_+ (\dirac_H^{\Rank\maE}) - \maP_1^{\Rank\maE} \right).
\end{multline*}
Now the path $(\maP_v^\maE)_{0\leq v \leq 1}$  is a spectral section for the family $(\dirac^\maE_{v H})_{0\leq v \leq 1}$. As already mentioned, the spectral flow $\sf ( (\dirac^\maE_{v H})_{0\leq v \leq 1}) $ can be easily proved to coincide with 
$$
 \Tr \left( \pi_+ (\dirac_H^\maE) - \maP_1^\maE\right)  -  \Tr \left( \pi_+ (\dirac^{\maE}) - \maP_0^{\maE})\right)
 $$
 for any choice of global spectral section. The same holds for the trivial bundle $\Rank\maE$. 
 Hence the conclusion.

We now prove the second item using the first one. When the metric $g$ on $Y$  has  positive scalar curvature, the index class $\ind_a(\dirac)$ vanishes in $K_1(C^*\Gamma)$ so we can apply the first item. By Corollary \ref{bismut-vanishing}, there exists a constant multiple metric $g'$ of $g$  (so in particular $g'$ lives in the conformal class of $g$) such that  for any $r\geq 1$, the operator $\dirac_{H}^{rg'}$ is invertible as a regular operator on the  Hilbert module $\maH_Y$.  Moreover, for such $r\geq 1$, the  path $(\dirac^{rg'}_{vH})_{0\leq v\leq 1}$ is composed of invertible operators. If we then apply the first item for the data associated with the metric $rg'$, we get for any choice of global spectral section $\maP$ now for the path $(\dirac^{r g'}_{vH})_{0\leq v \leq 1}$: 
\begin{multline*}
\rho_{\rm {spin}} (Y, \maE, H, rg') - \rho_{\rm {spin}} (Y, \maE, rg') = 
 \left[\Tr \left( \pi_+ (\dirac_{H}^{rg', \maE}) - \maP_1^\maE\right)  -  \Tr \left( \pi_+ (\dirac^{rg', \maE}) - \maP_0^{\maE})\right)\right] \\ - 
\left[\Tr \left( \pi_+ (\dirac_{H}^{rg', \Rank\maE}) - \maP_1^{\Rank\maE}\right)  -  \Tr \left( \pi_+ (\dirac^{rg', \Rank\maE}) - \maP_0^{\Rank\maE})\right)\right]
\end{multline*}
Now since the Hilbert spaces operators $\dirac_{vH}^{rg', \maE}$, $\dirac^{rg', \maE}$, $\dirac_{vH}^{rg', \Rank\maE}$ and $\dirac^{rg', \Rank\maE}$ are obtained using the composition operation from the Hilbert modules operators $\dirac_{vH}$ and $\dirac$, they are invertible and we know that 
$$
 \Tr \left( \pi_+ (\dirac_{H}^{rg', \maE}) - \maP_1^\maE\right)  -  \Tr \left( \pi_+ (\dirac^{rg', \maE}) - \maP_0^{\maE})\right) = 
 \sf ( (\dirac^{rg', \maE}_{v H})_{0\leq v \leq 1}) = 0
 $$
and 
$$
\Tr \left( \pi_+ (\dirac_{H}^{rg', \Rank\maE}) - \maP_1^{\Rank\maE}\right)  -  \Tr \left( \pi_+ (\dirac^{rg', \Rank\maE}) - \maP_0^{\Rank\maE})\right) =  \sf ((\dirac^{rg', \Rank\maE}_{v H})_{0\leq v \leq 1}) = 0.
$$
These latter two spectral flows coincide with the corresponding APS spectral flows (algebraic net number of eigenvalues crossing $0$), they are trivial since  the paths are both composed of invertible operators. This ends the proof of Theorem \ref{Main} since we know that
$$
 \rho_{\rm {spin}} (Y, \maE, rg') =  \rho_{\rm {spin}} (Y, \maE, g).
$$

\end{proof}

%
%\begin{corollary} Assume that $\ind_a(\dirac)$ vanishes in $K_1(C^*\Gamma)$, then the difference of spectral flows $
%\SF ( (\dirac^\maE_{v H})_{0\leq v \leq 1}) - \SF ((\dirac^{\Rank\maE}_{v H})_{0\leq v \leq 1})$
%is a conformal invariant which vanishes when the metric $g$ is conformally equivalent to a metric with positive scalar curvature.
%\end{corollary}
%
%\begin{proof}
%We apply the first item of Theorem \ref{Main} and the conformal invariance of the twisted rho invariant proved in \cite{BM-JGP} to deduce the first statement. The second one is deduce from the second item of Theorem \ref{Main}.
%\end{proof}

\begin{corollary}\label{cor-main}
Assume that the metric $g$ on $Y$ has positive scalar curvature, that the group $\Gamma$ is torsion free and satisfies the maximal Baum-Connes conjecture, then for any $\maE$ and $H$ as before, 
$$
\rho_{spin}(Y,\E,H, r g) = 0 \text{  for all } r\gg 0.
$$ 
In particular, there exists a metric $g'$ in the conformal class of $g$  such that  for any $\lambda\geq 1 $, we have
$$
\rho_{\rm {spin}} (Y, \maE,  H, \lambda  g') = 0.
$$
and so $\rho_{\rm {spin}} (Y, \maE,  H,  g') = 0$.
\end{corollary}

\begin{proof}
Apply Theorem \ref{Main} and deduce that 
$$
\rho_{spin}(Y, \E, H, g') = \rho_{spin}(Y, \E, g)
$$
%Now under the assumptions of the corollary, 
Then apply the main result in \cite{Keswani} and deduce that $\rho_{spin}(Y, \E)=0$, hence the conclusion.
\end{proof}

\begin{remark}
From the previous proof, we easily deduce that under the assumptions of Corollary \ref{cor-main} and for any metric $g''$ in the conformal class of $g$, we have with the notations of Section \ref{sect:bismut}, $
\rho_{spin}(Y, \E,  H_{g''/g'}, g'') = 0.$
\end{remark}

\begin{remark}
When $H$ is a closed degree $3$ real differential form, we prove a stronger result in the appendix which uses Proposition 3.1  in \cite{BM-JGP} and an index theorem from \cite{M92}. The hypothesis is then reduced to the fact that the particular representation associated with $\maE$ be connected to the trivial representation. However, this method apparently cannot be generalized to higher degree forms $H$ without proving a local index formula extending the Bismut formula in \cite{Bismut}.
\end{remark}

\appendix

\section{Vanishing of twisted rho using the representation variety}

In this section we give an alternate proof of the vanishing of the rho invariant of the twisted Dirac operator with coefficients in a flat Hermitian vector bundle on a compact odd dimensional Riemannian spin manifold $Y$ of positive scalar curvature in the special case of closed degree 3 differential form twists $H$, using the representation variety of $\Gamma = \pi_1(Y)$ instead. It adapts the approach in \cite{M92}.

Let $\fR={\rm Hom} (\Gamma, U(p))$ denote the representation variety of $\Gamma$. 
We now construct a generalization of the Poincar\'e vector bundle $\cP$ over $\fR$. Let $E\Gamma \to B\Gamma$ be a principal $\Gamma$-bundle over the space $B\Gamma$ with contractible total space $E\Gamma$. Then $B\Gamma$ is called the classifying space of the group $\Gamma$. Let $f \colon Y	\to B\Gamma$ be a continuous map classifying the universal $\Gamma$-covering of $Y$. We construct a tautological rank $p$ Hermitian vector bundle 
$\cP$ over $B\Gamma \times \fR$ as follows: consider the action of $\Gamma$ on $E\Gamma \times \fR \times \C^p$ 
given by
\begin{align*}
E\Gamma \times \fR  \times \C^p \times \Gamma &\longrightarrow E\Gamma \times \fR  \times \C^p\\
((q,\sigma, v), \tau) & \longrightarrow (q\tau ,\sigma, \sigma(\tau)v)
\end{align*}
Define the universal rank $p$ Hermitian vector bundle $\cP$ over $B\Gamma \times \fR$ to be the quotient $(E\Gamma \times \fR  \times \C^p)/\Gamma$. $\cP$ has the property that the restriction $\cP\big|_{B\Gamma \times \sigma}$ is the flat Hermitian vector bundle over $B\Gamma$ defined by $\sigma$. Let $I$ denote the closed unit interval $[0,1]$ 
and $\gamma: I \to \fR$ be a smooth path in $\fR$ joining the unitary representation $\sigma$ to the trivial representation. Define $\E = (f \times \gamma)^*\cP \to Y \times I$ to be the Hermitian vector bundle over $Y \times I$. 
By the Kunneth Theorem in cohomology, we have $\ch(\cP) =	\sum_i x_i \xi_i$, where $\ch(\cP)$ is the Chern character of $\cP$, $x_i \in H^*(B\Gamma, \R)$ and $\xi_i \in H^*(\fR, \R)$. It follows that if $y_i=	f^*(x_i)$ and $\mu_i=\gamma^*(\xi_i)$ ,then $\ch(\E) = \sum_i y_i\mu_i$. Note that the pullback connection makes $\E$ into a Hermitian vector bundle over 
$Y \times I$. 

\begin{theorem}	
Let $(Y, g)$ be a smooth compact connected oriented spin manifold of odd dimension, where $g$ is a Riemannian metric of positive scalar curvature on $Y$ and $H$ a closed degree 3 differential form on $Y$. Let $\Gamma$ denote the fundamental group of $Y$ and $\sigma_\E: \Gamma \to U(p)$ a unitary representation that can be connected by a smooth path $\gamma : I \to \fR$ to the trivial  representation in
the representation space $\fR$. Then for all  $r \gg 0$, $\rho_{spin}(Y, \E, H, rg) =0$.
\end{theorem}

\begin{proof} Observe that the unitary connection induced on $\E$ has curvature which is a multiple of 
$dt$, so that $\ch(\E) = p + c_1(\E)$, where $c_1(\E)$ is the first Chern class of $\E$ and 
$t$ is the variable on the interval $I$. It follows that $\ch(\E) = p + y\mu$	where $y\in H^1(Y, \R)$ and 
$\mu\in H^1(I,\R)$, as $c_1(\E)$ can be represented by the trace of the curvature of a unitary connection on $\E$. 
Since $c_1(\E) = (f \times \gamma)^*c_1(\cP)$, we see that $y = f^*(x)$ and $\mu=\gamma^*(\xi)$ for some
$x\in H^1(B\Gamma,\R)$ and $\xi \in H^1(\fR, \R)$. Let  $\E_t$ denote the flat hermitian bundle over $Y$
 determined by the representation $\gamma(t):\Gamma \to U(p)$.
By the APS index theorem for manifolds with boundary \cite{APS1}, combined with Bismut's local index theorem \cite{Bismut} that determines the 
integrand for the $H$-twisted Dirac operator, we see that
$$\Index (D^\E_H, P) = \int_{Y\times I} \widehat A(Y)\ch(\E) - \rho_{spin}(Y,\E, H, g)$$
$Y\times I$ since $\widehat A(Y \times I) = \widehat A(Y)$. Here $\widehat A(Y)$ is the 
A-hat characteristic class of $Y$. 
$\Index (D^\E_H, P)$ is the index of the $H$-twisted
Dirac operator on $Y\times I$ with global APS boundary conditions.
From the
discussion above
$$ \int_{Y\times I} \widehat A(Y)\ch(\E)=  \int_Y \widehat A(Y) y \int_I \mu.$$
By the results of \cite{APS3} and \cite{Melrose}, it follows that the spectral flow of the path of $H$-twisted Dirac operators
$[I\ni t \mapsto \dirac^{g, \E_t}_H]$ is equal to $\Index (D^\E_H, P)$.
Therefore
$$
{\rm sf}[I\ni t \mapsto \dirac^{g, \E_t}_H]= \int_Y \widehat A(Y) y \int_I \mu - \rho_{spin}(Y,\E, H, g)
$$
Since $(Y,g)$ is a spin Riemannian manifold of positive scalar curvature, it follows from the work of Gromov and Lawson \cite{GrLaw} that 
$\displaystyle \int_Y \widehat A(Y)f^*(x) = 0$
for all $x\in H^1(B\Gamma,\R)$. Since $g$ is a metric of positive scalar curvature, then by Proposition \ref{bismut-vanishing}, we see that for all $r\gg 0$ and for any $t\in I$,   $\dim(\ker\dirac^{rg, \E_t}_{H})=0$. This implies that the spectral flow of the family $[I\ni t \mapsto \dirac^{rg, \E_t}_{H}]$ is zero for any $r\gg 0$. 

Therefore we conclude that $ \rho_{spin}(Y,\E, H, rg)=0$ for all $r\gg 0$ as claimed.
\end{proof}

\begin{remark}
Notice that the maximal Baum-Connes conjecture for $C^*\Gamma$ implies the connectedness of the unitary dual of $\Gamma$. 
A comparable assumption would be that the representation variety $\fR$ of $\Gamma$ is connected.
\end{remark}

\end{document}